\documentclass[a4paper,11pt,reqno]{amsart}

\usepackage{amsmath}%
\allowdisplaybreaks
\usepackage{amsfonts}%
\usepackage{amssymb}%
\usepackage{graphicx}
\usepackage{float}
\usepackage{amssymb}
\usepackage{color}
\usepackage{dsfont}
\usepackage{setspace}
\usepackage{cite}
\theoremstyle{plain}

\newtheorem{theorem}{Theorem}
\newtheorem{definition}[theorem]{Definition}

\newtheorem{proposition}[theorem]{Proposition}
\newtheorem{remark}[theorem]{Remark}
\newtheorem{corollary}[theorem]{Corollary}

\def\N{\mathbb{N}}

\def\R{\mathbb{R}}

\def\P{\mathcal{P}}

\def\ex{\mathbb{E}}

\def\eins{\mathds{1}}
\def\lin{{\rm lin}}
\def\ve{\varepsilon}

\numberwithin{equation}{section}

\begin{document}

\title[Optimal approximation of Skorohod integrals]{Optimal approximation of Skorohod integrals}\label{}
 
 \author[A. Neuenkirch and P. Parczewski]{Andreas Neuenkirch and Peter Parczewski
%\footnote{Author footnote.}
}
%\index[aindx]{Author, F.} % or \aindx{Author, F.}
%\index[aindx]{Author, S.} % or \aindx{Author, S.}

\address{University of Mannheim, Institute of Mathematics, A5, 6, D-68131 Mannheim, Germany.}
\email{neuenkirch@math.uni-mannheim.de, parczewski@math.uni-mannheim.de}

\date{\today}

\begin{abstract}
In this manuscript, we determine the optimal approximation rate for Skorohod integrals of sufficiently regular integrands. This generalizes the optimal approximation results for It\^o integrals. However, without adaptedness and the It\^o isometry, new proof techniques are required. The main tools are a characterization via S-transform and  a reformulation of the Wiener chaos decomposition in terms of Wick-analytic functionals. 
\end{abstract}

%We also provide Skorohod Wagner-Platen schemes and derive a {\sf numerical scheme, which attains the optimal


\maketitle

\section{Introduction}

In several applications, e.g. the computation of derivative-free option price sensitivities

\cite{Fournie, Chen_Glasserman} or the payoff-smoothing in mathematical finance \cite{AN},  Skorohod integrals of the type
\begin{align*}
I=\int_0^1 u_s dW_s 
\end{align*}
arise. Here $u=(u_t)_{t \in [0,1]}$ is a possibly non-adapted process and $W=(W_t)_{t \in [0,1]}$ is a Brownian motion. The Skorohod integral is an extension of the standard It\^o integral, see e.g. \cite{Nualart, Di_Nunno_und_so, Holden_Buch} and Section \ref{Section_Skorohod_integrals}. The standard strategy to numerically deal with these expressions has been to rewrite the integral with an integration by parts formula, see e.g. Proposition 1.3.3 in \cite{Nualart} or Theorem 3.15 in \cite{Di_Nunno_und_so} , which (hopefully) leads to a simpler expression involving It\^o integrals instead of Skorohod integrals, and to discretize these integrals then.

However, the best possible rate of convergence for the $L^2$-approximation of $I$, given the knowledge of the integrand $u$ and a finite number of evaluations of $W$, has not been analysed so far. This motivates us to study the following question: What is the optimal convergence rate for the $L^2$-approximation of
 \begin{align} \label{eq_Skorohod_int_intro} 
 I=\int_0^1 f(s, W_s, W_{\tau_2}, \ldots, W_{\tau_K}) dW_s, 
 \end{align}
where $\tau_2, \ldots, \tau_K \in [0,1]$ are fixed timepoints of nonadaptedness, given complete information of $f: \mathbb{R}^K \rightarrow \mathbb{R}$ and knowledge
of $W_{1/n}, \ldots, W_{1}$, $W_{\tau_2}, \ldots, W_{\tau_K}$ ? Clearly the optimal approximation is
 \begin{align} \label{eq_Skorohod_int_intro_opt} 
 \hat{I}^n = \mathbb{E} \left[\left. \int_0^1 f(s, W_s, W_{\tau_2}, \ldots, W_{\tau_K}) dW_s \right|W_{\frac{1}{n}}, W_{\frac{2}{n}}, \ldots, W_1, W_{\tau_2}, \ldots, W_{\tau_K}\right], 
 \end{align}
 thus it remains to determine
 \begin{align} \label{eq_Skorohod_int_intro_opt_rate}  
 e_n := \ex[(I-\hat{I}^n)^2]^{1/2}. 
 \end{align}
 Obviously, integral \eqref{eq_Skorohod_int_intro} contains only a finite and fixed nonadapted part $W_{\tau_2}, \ldots, W_{\tau_K}$. This is much simpler than the original problem for arbitrary nonadapted processes, but its analysis will give us an indication, which convergence rates are best possible for the general problems.
 
 Our findings are as follows: Under some smoothness and growth conditions on $f$ (see Theorem \ref{thm_optimal_approx_simple_Skorohod}), we obtain the asymptotic behaviour
 \begin{align}\label{eq:ConvRateIntro}
 e_n \approx   \frac{1}{\sqrt{12}}\cdot C(f,\tau_2, \ldots, \tau_k) \cdot n^{-1} 
 \end{align}
 with 
 $$ C(f,\tau_2, \ldots, \tau_k)= \left(\int_{0}^{1} \ex[\mathcal{L} f^{\diamond}(s, \tau_2,\ldots, \tau_K,W_s,W_{\tau_2}, \ldots, W_{\tau_K})^2] ds\right)^{1/2},$$
where $f^{\diamond} : [0,1]^K \times \mathbb{R}^K \rightarrow \mathbb{R}$ arises from the Wick-analytic representation of $  f(s, W_s, W_{\tau_2}, \ldots, W_{\tau_K})$ 
via the chaos decomposition and
\begin{equation*}
\mathcal{L}:= \left(\sum_{1\leq k \leq K}\dfrac{\partial}{\partial t_k} + \frac{1}{2} \sum_{1\leq k,l \leq K} \dfrac{\partial^2}{\partial x_k \partial x_l}\right)
\end{equation*}
denotes the differential operator in the Skorohod It\^o formula, see Theorem \ref{thm_Ito_formula_Skorohod}.

For adapted integrands $u_s=f(s,W_s)$ with $f \in C^{1,2}$ we have $\mathcal{L} f^{\diamond} = \mathcal{L} f$ and \eqref{eq:ConvRateIntro} is a generalization of results on optimal approximation of It\^o integrals. The complexity of It\^o integration, i.e. the rate of convergence in \eqref{eq:ConvRateIntro}, is already established in \cite{Wasilkowski_Wozniakowski}. In \cite{Mueller_Gronbach} this is extended to the optimal approximation of solutions of stochastic differential equations. In the latter article the constant $C$ is determined via a differential operator $\mathcal{L}$ (for $K=1$). 
If $C(f,\tau_2, \ldots, \tau_k)=0$, then $I$ can be simulated exactly -- at least theoretically -- and  we have
\begin{align} \label{exact} I= \hat{I}^n \qquad \Longleftrightarrow \qquad C(f,\tau_2, \ldots, \tau_k)=0. \end{align}
This is in line with the analysis in \cite{Przybylowicz} for It\^o integrals.

%derivative operator and the pathwise integral in \eqref{eq_opt_constant} immediatelly generalizes the optimal approximation rate for It\^o integrals 

 %Similarly, the integral in \eqref{eq_opt_constant} appears in a simple extension of the It\^o formula for Skorohod integrals of the type \eqref{eq_Skorohod_int_intro} (see Theorem \ref{thm_Ito_formula_Skorohod}). 

In contrast to the It\^o case, the computation of \eqref{eq_Skorohod_int_intro_opt_rate} cannot be reduced to an It\^o isometry. However, due to a Wiener chaos expansion of the nonadapted processes and advantageous reformulations of Skorohod integrals in terms of Wick products, the lack of isometry is handled. Our assumptions on $f$ ($f \in C^{1;2,\ldots, 2}([0,1]\times \R^{K})$ and some integrability and H\"older growth conditions (see Section \ref{Section_Optimal_approx_Skorohod_integrals}) 
%{\tt ... which ... ? $C^{1,2}$ ? ...} 
are partially necessary for the existence of the Skorohod integral. These assumptions are weaker than in the work on optimal approximation of It\^o integrals. %{\tt ... really ? ...}
In particular, due to a convenient reformulation of \eqref{eq_Skorohod_int_intro} into Wick-analytic functionals, we are able to separate the dependence of $f$ on the dynamic evaluation $W_s$ and the nonadapted parts $W_{\tau_2}, \ldots, W_{\tau_K}$. Thanks to this we obtain several equivalent conditions, when an exact simulation of \eqref{eq_Skorohod_int_intro_opt_rate} is possible, from which we can deduce \eqref{exact}.
 The additional tools in the proofs are a characterization of random variables via S-transform, an intermediate value theorem for finite chaos terms and a diagonal argument to include infinite chaos elements.

%Motivated by the Wick-analytic representations, we introduce a Wagner-Platen scheme for Skorohod integrals in \eqref{eq_Skorohod_int_intro} which extends the Wagner-Platen scheme for It\^o integrals. Following the ideas in \cite{Przybylowicz} we modify it to a conditional Wagner-Platen scheme which exhibits the optimal approximation rate \eqref{eq_opt_constant}. This optimal scheme is reformulated into a Wick-free version which is applicable for possible implementations.

The paper is organized as follows: In Section \ref{Section_Skorohod_integrals} we give a self-contained introduction to Skorohod integrals of the type \eqref{eq_Skorohod_int_intro} and the fruitful refomulations of the Wiener chaos decomposition and conditional expectations in terms of Wick calculus. The section ends with a Skorohod It\^o formula for our approach. Section \ref{Section_Exact_sim} is devoted to the exact simulation of the Skorohod integral \eqref{eq_Skorohod_int_intro}. The main result on the optimal approximation of Skorohod integrals is the content of Section \ref{Section_Optimal_approx_Skorohod_integrals}. % Finally, in Section \ref{Section_approx_scheme} the Wagner-Platen and conditional Wagner-Platen schemes for \eqref{eq_Skorohod_int_intro} are investigated.

\section{Skorohod integrals}\label{Section_Skorohod_integrals}

We suppose a Brownian motion $(W_{t})_{t \in [0,1]}$ on the probability space $(\Omega, \mathcal{F}, P)$, where the $\sigma$-field $\mathcal{F}$ is generated by the Brownian motion and completed by null sets. Therefore the stochastic calculus is based on the Gaussian Hilbert space $\{I(f) : f \in L^2([0,1])\} \subset L^2(\Omega)$, where $I(f)$ denotes the Wiener integral. We denote the norm and inner product on $L^2([0,1])$ by $\|\cdot \|$ and $\langle \cdot, \cdot\rangle$.

We establish the Skorohod integral by the S-transform and make use of the connection to Wick calculus. An essential tool and, roughly speaking, our paradigm of studying Skorohod integrals, is the reformulation of processes into Wick-analytic versions via the Wiener chaos decomposition. Aiming at optimal approximation, we collect some basic properties of conditional expectations and Skorohod integrals. At the end of this section, we present an It\^o formula for Skorohod integrals. Its proof illustrates the advantage of using Wick-analytic representations and will be crucial to derive our results.

For every $f \in L^2([0,1])$, we denote 
\begin{equation}\label{eq_Wick_exp}
\exp^{\diamond}(I(f)) := \exp\left(I(f) - 1/2\|f\|^2\right)
\end{equation}
the \emph{Wick exponential}. Due to the generating function, we have
\begin{equation}\label{eq_Wick_exp_by Hermite_polynomials} \exp^{\diamond}(I(f)) = \sum_{k=0}^{\infty} \frac{1}{k!}\,h^k_{\| f \|^2}(I(f)), \end{equation}
with the Hermite polynomials
\begin{equation*} h^{k}_{\sigma^{2}}(x) = (-\sigma^2)^{k} \exp\left(\frac{x^2}{2\sigma^2}\right) \frac{d^k}{dx^k} \exp\left(\frac{- x^2}{2\sigma^2}\right). \end{equation*}
In particular, the Wick exponentials exhibit the following renormalization and integrability properties. For a proof of these basic facts we refer to Theorem 3.33 and the Corollaries 3.37, 3.38, 3.40 in \cite{Janson}.

\begin{proposition}\label{prop_Wickexp_properties}
For $f,g \in L^2([0,1])$, $p>0$, we have:
\begin{enumerate}
 \item $\exp^{\diamond}(I(f)) = \frac{\exp(I(f))}{\ex[\exp(I(f))]}$,
 \item $\ex[\exp^{\diamond}(I(f)) \exp^{\diamond}(I(g))] = \exp(\langle f,g\rangle)$,
 \item $\ex[(\exp^{\diamond}(I(f)))^p] = \exp\left(\frac{p(p-1)}{2} \|f\|^2\right)$,
 \item The set $\{\exp^{\diamond}(I(f)) : f \in L^2([0,1])\}$ is total in $L^2(\Omega)$.
\end{enumerate}
\end{proposition}

%\begin{proof} (i)-(iii) are clear by computation on \eqref{eq_Wick_exp}, for (iv) see e.g. \cite[Corollary 3.40]{Janson}. \end{proof}

For every $X \in L^2(\Omega, \mathcal{F}, P)$ and $h \in L^2([0,1])$, the \emph{S-transform} of $X$ at $h$ is defined as
\[
(SX)(h) := \ex[X \exp^{\diamond}(I(h))]. 
\]

For every $X \in L^2(\Omega, \mathcal{F}, P)$, $(SX)(\cdot)$ is a continuous function on $L^2([0,1])$. Moreover,  due to Proposition \ref{prop_Wickexp_properties} (iv), $S$ is an injective continuous linear map from the space $L^2(\Omega, \mathcal{F}, P)$ into the space of (analytic) functions on $L^2([0,1])$ (see e.g. \cite[Chapter 16]{Janson} for more details).
%In fact it is an isometry of $L^2(\Omega, \mathcal{F}, P)$ onto a real reproducing kernel Hilbert space on $\{I(f) : f \in L^2([0,1])\}$ with the kernel $K(x,y) = \exp(\ex[xy])$ (\cite[Theorem 16.11, Theorem 16.21]{Janson}). 

As an example, for $f,g \in L^2([0,1])$, we have
\begin{equation*}
(S \ \exp^{\diamond}(I(f)))(g) = \exp\left(\langle f,g\rangle\right).
\end{equation*}

The extension of the It\^{o} integral to nonadapted integrands is the Skorohod integral. Besides the definitions of the Skorohod integral via multiple Wiener integrals or as the adjoint of the Malliavin derivative (cf. \cite{Nualart, Janson}), there is a simple introduction via S-transforms (cf. e.g. \cite[Section 16.4]{Janson}):

\begin{definition}
Suppose $u = (u_t)_{t \in [0,1]} \in L^2(\Omega \times [0,1])$ is a (possibly nonadapted) square integrable process on $(\Omega, \mathcal{F}, P)$ and $Y \in L^2(\Omega, \mathcal{F}, P)$ such that
\[
\forall g \in L^2([0,1]): \ (S Y)(g) = \int_{0}^{1} (S u_s)(g) g(s) ds, 
\]
then $\int_{0}^{1}u_s dW_{s} = Y$ defines the Skorohod integral of $u$ with respect to the Brownian motion $W$. 
\end{definition}

For more information on the Skorohod integral we refer to \cite{Janson} or \cite{Kuo}.

We recall that for every $k \in \N$ the $k$-th Wiener chaos $H^{:k:}$ is the completion of $\{h^k_{\| f \|^2}(I(f)) : f \in L^2([0,1])\}$ in $L^2(\Omega)$ and these subspaces are orthogonal and fulfill $L^2(\Omega, \mathcal{F}, P) = \bigoplus_{k \geq 0} H^{:k:}$. Thus, for every random variable $X \in L^2(\Omega)$ we denote its \emph{Wiener chaos decomposition} as
\[
X=\sum_{k=0}^{\infty} \pi_k(X),
\]
for the projections $\pi_k: L^2(\Omega) \rightarrow H^{:k:}$. We refer to \cite{Janson, Holden_Buch} for further details and a reformulation in terms of multiple Wiener integrals. We recall that a process $u \in L^2(\Omega \times [0,1])$ is Skorohod integrable if and only if
\[
\sum_{k=0}^{\infty} (k+1) \|\pi_k(u_s)\|^2_{L^2(\Omega \times [0,1])} < \infty. 
\]

The S-transform is closely related to a product imitating uncorrelated random variables as $\ex[X \diamond Y]=\ex[X]\ex[Y]$. As a consequence of Proposition \ref{prop_Wickexp_properties} and the injectivity of the S-transform, we have

\begin{proposition}\label{prop_Wick_product}
 We define the \emph{Wick product}\index{$\diamond$} by
\begin{align*}
\mathcal{D}^{\diamond} := &\left\{(X,Y) \in L^{2}(\Omega) \times L^{2}(\Omega) :\right.
 \\
 & \ \ \exists Z_{X,Y} \in L^{2}(\Omega) \ \forall g \in L^2([0,1]) \ \left.
(S Z_{X,Y})(g) = (S X)(g) (S Y)(g)
\right\}  \\ \diamond:&\; \mathcal{D}^{\diamond} \rightarrow L^{2}(\Omega),\quad (X,Y) \mapsto Z_{X,Y}.
\end{align*}
Then $\mathcal{D}^{\diamond}$ is a dense subset of $L^{2}(\Omega) \times L^{2}(\Omega)$,  the Wick product $\diamond$ is well-defined, bilinear and closed on $\mathcal{D}^{\diamond}$. In particular, $(\exp^{\diamond}(I(f)), \exp^{\diamond}(I(g))\in \mathcal{D}^{\diamond}$ for all $f,g \in L^2([0,1])$ and the functional equation
\begin{equation*}\label{eq_Wickpro_Wick_exp}
\exp^{\diamond}(I(f)) \diamond \exp^{\diamond}(I(g)) = \exp^{\diamond}(I(f+g))
\end{equation*}
 is valid. 
\end{proposition}

For more details on the Wick product we refer to \cite{Janson, Holden_Buch}. The Wick calculus is a fundamental tool in stochastic analysis (cf. \cite[Chapter 7]{Janson}, \cite{Kuo,Holden_Buch}) and closely related to the Skorohod integral, cf. Proposition \ref{prop_S_transform_properties}.

For a proof of the following reformulation of Wick products of Gaussian random variables in terms of generalized Hermite polynomials (see e.g. \cite{Avram_Taqqu}) we refer to \cite[Chapter 3]{Janson}, cf. \cite[Lemma 2.1]{P2}:

\begin{proposition}\label{prop_Wick_product_by_Wick_exp}
Suppose $f_1, \ldots, f_K \in L^2([0,1])$, $l_1,\ldots, l_K \in \N$. Then there exists a polynomial $h:\R^K \rightarrow \R$ such that
\[
h(x_1,\ldots, x_K) =  \dfrac{\partial^{\sum l_i}}{\partial a_1^{l_1} \cdots \partial a_K^{l_K}} \left.\exp\left(\sum a_i x_i - \frac{1}{2} \|\sum a_i f_i\|^2\right)\right|_{a_1=\ldots =a_K=0}, 
\]
and
\[
I(f_1)^{\diamond l_1} \diamond \cdots \diamond I(f_K)^{\diamond l_K} = h(I(f_1), \ldots, I(f_K)).
\]
In particular, for fixed  $t_1, \ldots, t_K \in [0,1]$, we have the  representation
%\begin{equation}\label{eq_Wick_product_by_Wick_exp} I(f_1)^{\diamond l_1} \diamond \cdots \diamond I(f_K)^{\diamond l_K} =h(I(f_1),\ldots, I(f_K)),\end{equation}\cb{where the right hand side depends on the covariance matrix $\langle f_i,f_j \rangle_{1\leq i,j \leq K}$ as well. For fixed  $t_1, \ldots, t_K \in [0,1]$, we obtain
\begin{align}\label{eq_Wick_product_by_Wick_exp}
W_{t_1}^{\diamond l_1} \diamond \cdots \diamond W_{t_K}^{\diamond l_K} = h^{\diamond}(t_1,\ldots, t_K, W_{t_1}, \ldots, W_{t_K})
\end{align}
for some polynomial $h^{\diamond}: [0,1]^K \times \R^K \rightarrow \R$.

\end{proposition}
The Wick exponential and Hermite polynomials of Gaussian random variables in \eqref{eq_Wick_exp} and \eqref{eq_Wick_exp_by Hermite_polynomials} are reformulated as
\begin{equation}\label{eq_Wick_exp_by_Wick_powers}
\exp^{\diamond}(I(f)) = \sum\limits_{k=0}^{\infty}\frac{1}{k!} I(f)^{\diamond k}. 
\end{equation}

\begin{remark}\label{remark_Wiener_chaos_by_Wick_products}
We recall the following reformulation of the Wiener chaos expansion. For fixed  $t_1, \ldots, t_K \in [0,1]$, $f: \R^K \rightarrow \R$ and a square integrable left hand side, we have a representation
\begin{equation}\label{eq_Wiener_chaos_by_Wick_products}
f(t,W_{t_1}, \ldots, W_{t_K}) = \sum_{l_1, \ldots, l_K\geq 0} a_{l_1, \ldots, l_K}(t) W_{t_1}^{\diamond l_1} \diamond \cdots \diamond W_{t_K}^{\diamond l_K}, \qquad t \in [0,1],
\end{equation} 
where $a_{l_1, \ldots, l_K}:  [0,1] \rightarrow \R$ for all $l_1, \ldots, l_K\geq 0$.
Given such a square integrable random variable $f(t,W_{t_1}, \ldots, W_{t_K})$, the expansion \eqref{eq_Wiener_chaos_by_Wick_products} is obtained as follows: W.l.o.g. let $t_1 < \ldots < t_K$. Then the random variables $W_{t_1}, W_{t_2}-W_{t_1}, \ldots, W_{t_K}- W_{t_{K-1}}$ are orthogonal in $L^2(\Omega)$ and all Wick products of these elements constitute an orthogonal basis of the space of square integrable $\sigma(W_{t_1}, \ldots, W_{t_K})$-measurable random variables (cf. \cite[Theorem 3.21]{Janson})). Hence, an expansion with respect to this basis and the bilinearity of the Wick product yields a representation \eqref{eq_Wiener_chaos_by_Wick_products}.

%In particular, given a square integrable $f(t_1,W_{t_1}, \ldots, W_{t_K})$, via \eqref{eq_Wiener_chaos_by_Wick_products} and \eqref{eq_Wick_product_by_Wick_exp}, we conclude a function $f^{\diamond}: [0,1]^K \times \R^K \rightarrow \R$ and a representation\begin{align}\label{eq_f_diamond_repr} f(W_{t_1}, \ldots, W_{t_K}) = f^{\diamond}(t_1, \ldots, t_K,W_{t_1}, \ldots, W_{t_K}).
\end{remark}

\begin{remark}\label{remark_sine}
We notice that in many cases the representation \eqref{eq_Wiener_chaos_by_Wick_products} can be simply obtained by reduction on Hermite polynomials or Wick exponentials. As an example, by a reduction on the Wick-sine of a Gaussian random variable $X$ as
\begin{equation}\label{eq_Wick_sine}
\sin^{\diamond}(X) = \sin(X)\exp^{\frac{1}{2}\|X\|_{L^2}^2} = \sum_{M=0}^{\infty} \frac{(-1)^{M-1}}{(2M-1)!} X^{\diamond (2M-1)},  
\end{equation}
cf. \cite[p. 107]{Holden_Buch} and the bilinearity of the Wick product, we observe
\begin{align*}
 &\sin\left(t\sum_{i=1}^{K} W_{t_i}\right) = \exp\left(-\frac{t^2}{2}\|\sum_{i=1}^{K} W_{t_i}\|_{L^2}^2\right)\sin^{\diamond}\left(t\sum_{i=1}^{K} W_{t_i}\right)\\
 &= \exp\left(-\frac{t^2}{2}\|\sum_{i=1}^{K} W_{t_i}\|_{L^2}^2\right) \sum_{M=0}^{\infty} \frac{(-1)^{M-1}t^{2M-1}}{(2M-1)!} \sum\limits_{\sum l_i = 2M-1} \binom{2M-1}{l_1, \ldots, l_K}W_{t_1}^{\diamond l_1} \diamond \cdots \diamond W_{t_K}^{\diamond l_K},
\end{align*}
i.e. the representation \eqref{eq_Wiener_chaos_by_Wick_products} with
\[
a_{l_1, \ldots, l_K}(t) =  \eins_{\{\sum l_i \in 2\N+1\}} \frac{(-1)^{(\sum l_i-1)/2}t^{2M-1}}{\sum l_i}\binom{\sum l_i}{l_1, \ldots, l_K}\exp\left(-\frac{t^2}{2}\|\sum_{i=1}^{K} W_{t_i}\|_{L^2}^2\right).
\] 
\end{remark}

In particular we have the following properties of S-transforms and Skorohod integrals:

\begin{proposition}\label{prop_S_transform_properties}
Suppose $X \in L^2(\Omega)$ and a (Skorohod integrable) process  $u\in L^2(\Omega \times [0,1])$. Then:
\begin{align*}
&(i) \ \forall g \in L^2([0,1]) \ (SX)(g)= \sum\limits_{k=0}^{\infty} (S \pi_k(X))(g),\\
&(ii) \ \forall g \in L^2([0,1]) \ \int_{0}^{1} (S u_s)(g) g(s) ds = \sum_{k \geq 0} \int_{0}^{1} (S \pi_k(u_s))(g) g(s) ds,\\
&(iii) \ \int_{0}^{1} X \diamond u_s dW_s = X \diamond \int_{0}^{1} u_s dW_s, \quad \textnormal{if both sides exist in } L^2(\Omega)\\
&(iv) \ \int_{0}^{1} X \diamond u_s ds = X \diamond \int_{0}^{1} u_s ds , \quad \textnormal{if both sides exist in } L^2(\Omega).
\end{align*}
\end{proposition}

\begin{proof}
For (i) and (ii) see e.g. \cite[Theorem 16.16 (ii)]{Janson} and \cite[Lemma 16.48]{Janson}. The statement (iii) for finite chaoses $\pi_k(u_s)$ follows analogously to the proof of \cite[Theorem 16.51]{Janson} from the definition of the Skorohod integral and Proposition \ref{prop_Wick_product}. Then,  by (i) and (ii), we conclude the assertion. Finally, an application of the S-transform and Fubini's theorem yields (iv).
\end{proof}

%The example
%\begin{equation}\label{eq_signum_Skorohod}
% \int_{0}^{1}sign(W_1) dW_s = sign(W_1) \diamond W_1 \notin L^2(\Omega)
%\end{equation}eq_f_diamond_repr
%shows that even simple square integrable random variables can fail the Skorohod integrability (cf. \cite[p. 65]{Kuo}).  

Dealing with $L^2$-norms of Gaussian random variables, we will frequently make use of
\begin{equation}\label{eq_Wick_product_of_Gaussian_inner_product}
\ex\left[\left(I(f_1) \diamond \cdots \diamond I(f_n)\right) \left(I(g_1) \diamond \cdots \diamond I(g_m)\right)\right] = \delta_{n,m} \sum\limits_{\sigma \in \mathcal{S}_{n}} \prod\limits_{i=1}^{n} \langle f_i, g_{\sigma(i)}\rangle,
\end{equation}
for all $n,m \in \N$ $f_1,\ldots, f_n, g_1,\ldots, g_m \in L^2([0,1])$, where $\mathcal{S}_{n}$ denotes the group of permutations on $\{1, \ldots, n\}$ (see e.g. \cite[Theorem 3.9]{Janson}). In particular, this implies
\begin{equation*}\label{eq_Wick_product_of_Gaussian_L_2_norm_inequality}
\ex\left[\left(I(f_1)^{\diamond k_1} \diamond \cdots I(f_n)^{\diamond k_n}\right)^2\right] \leq n! \prod_{i=1}^n k_i! \|f_i\|^{2k_i},
\end{equation*}
(cf. \cite[Proposition 3.1]{P2}). The infinite chaos random variables are extended to the following class of random variables (cf. \cite{Buckdahn_Nualart}): 

\begin{definition}\label{def_Wick_analytic_functionals}
We define the class of \emph{Wick-analytic functionals} as 
\begin{equation}\label{eq_Wick_analytic_functional}
F^{\diamond}(I(f_1), \ldots, I(f_K)) = \sum\limits_{k=0}^{\infty}a_{1,k} I(f_1)^{\diamond k} \diamond \cdots \diamond \sum\limits_{k=0}^{\infty} a_{K,k} I(f_K)^{\diamond k}, \ \  \max_{i \leq K}\sup\limits_{k} \sqrt[k]{k! |a_{i,k}|}  < \infty. 
\end{equation}
\end{definition}

These Wick analytic functionals are very close to the finite chaos elements for these reasons:

\begin{proposition}\label{prop_Wick_analytic_functionals_nice_properties}
Wick-analytic functionals in the sense of \eqref{eq_Wick_analytic_functional} fulfill for all $p \in \N$:
\begin{enumerate}
\item All moments are finite, i.e. $F^{\diamond}(I(f_1), \ldots, I(f_K)) \in L^p(\Omega)$.
 \item All Wick products of (a finite number of) Wick-analytic functionals exist in $L^p(\Omega)$. 
\end{enumerate}
\end{proposition}

\begin{proof}
(i): \ Thanks to the bound in \eqref{eq_Wick_analytic_functional}, \eqref{eq_Wick_product_of_Gaussian_inner_product} and Proposition \ref{prop_Wickexp_properties} (iii), for $C=\max\limits_{i \leq K}\sup\limits_{k} \sqrt[k]{|k! a_{i,k}|}$ and $p \in 2 \N$, we have
\begin{align*}
\displaybreak[0]
0&\leq \ex[\left(F^{\diamond}(I(f_1), \ldots, I(f_K))\right)^{p}]\\
&= \ex\left[\left(\sum_{(k_1^1,\ldots), \ldots, (k_1^p, \ldots)} a_{1,k_1^1} \cdots a_{K, k_K^p} (I(f_1)^{\diamond k_1^1} \diamond \cdots I(f_K)^{\diamond k_K^1}) \cdots (I(f_1)^{\diamond k_1^p} \diamond \cdots I(f_K)^{\diamond k_K^p})\right)\right]\\
&= \sum_{(k_1^1,\ldots), \ldots, (k_1^p, \ldots) \in \N^K} a_{1,k_1^1} \cdots a_{K,k_K^p} \ex\left[(I(f_1)^{\diamond k_1^1} \diamond \cdots I(f_K)^{\diamond k_K^1}) \cdots (I(f_1)^{\diamond k_1^p} \diamond \cdots I(f_K)^{\diamond k_K^p})\right]\\
&\leq \sum_{(k_1^1,\ldots), \ldots, (k_1^p, \ldots)} \frac{C^{k_1^1}}{(k_1^1)!} \cdots \frac{C^{k_K^p}}{(k_K^p)!} \ex\left[(I(|f_1|)^{\diamond k_1^1} \diamond \cdots I(|f_K|)^{\diamond k_K^1}) \cdots (I(|f_1|)^{\diamond k_1^p} \diamond \cdots I(|f_K|)^{\diamond k_K^p})\right]\\
&=\ex\left[\left(\exp^{\diamond} \left(C \sum_i I(|f_i|)\right)\right)^p\right] = \exp\left(\frac{p(p-1)}{2} C^2 \|\sum_i |f_i|\|^2\right)<\infty.
\end{align*}
Thus an application of the Lyapunov inequality yields (i). 

(ii): Since Wick products of Wick-analytic functionals are Wick-analytic functionals as well, we conclude the assertion from (i) and the H\"older inequality.
\end{proof}

We will mostly deal with simpler Wick-analytic functionals where $f_i=\mathbf{1}_{[0,t_i]}$.

\begin{proposition}\label{prop_Wick_product_Gaussian_analytic}
Let  $F^{\diamond}(W_{t_1}, \ldots, W_{t_K})$ be a Wick-analytic functional. Then there exists a function $f^{\diamond}: [0,1]^K\times \R^{K} \rightarrow \R$ which gives the 
%{... \tt ist $f^{\diamond}$ analytisch ?} 
analytic representation via \eqref{eq_Wick_product_by_Wick_exp} as
\begin{align}\label{eq_simple_Wick_analytic_functional}
f^{\diamond}(t_1, \ldots, t_K,W_{t_1}, \ldots, W_{t_K}) = F^{\diamond}(W_{t_1}, \ldots, W_{t_K}). 
\end{align}
Let $(k)_m := k(k-1) \cdots (k-m+1)$ be the falling factorial, $t_i \in [0,1]$ and $m_i \in \N$. Then:
\begin{enumerate}
 \item $\dfrac{\partial^{\sum m_i}}{\partial x_1^{m_1} \cdots \partial x_K^{m_K}} W_{t_1}^{\diamond l_1} \diamond \cdots \diamond W_{t_K}^{\diamond l_K} = \left(\prod_{i=1}^{K} (k_i)_{m_i}\right) W_{t_1}^{\diamond l_1-m_1} \diamond \cdots \diamond W_{t_K}^{\diamond l_K-m_K}$.
\item $f^{\diamond} \in C^{\infty}([0,1]^K\times \R^{K},\R)$. In particular, all derivatives of Wick-analytic functionals 
$F^{\diamond}(W_{t_1}, \ldots, W_{t_K})$ are Wick-analytic as well.
\end{enumerate}
\end{proposition}

\begin{proof}
The existence of $f^{\diamond}$ and (i) \ follow directly by \eqref{eq_Wick_product_by_Wick_exp}. Due to (i), the assertion (ii) is true for all finite chaoses. Similarly to Proposition \ref{prop_Wick_analytic_functionals_nice_properties} (i), we have
\begin{align*}
\displaybreak[0]
0&\leq \ex\left[\left(\dfrac{\partial^{\sum m_i}}{\partial x_1^{m_1} \cdots \partial x_K^{m_K}} F^{\diamond}(W_{t_1}, \ldots, W_{t_K})\right)^{2}\right]\\
&= \sum_{\substack{(k_1^1,\ldots), (k_1^2, \ldots) \in \N^K\\ k_i^1, k_i^2 \geq m_i}} a_{1,k_1^1} \, a_{1,k_1^2} \cdots a_{K,k_K^1} \, a_{K,k_K^2} \prod_{i=1}^{K} (k_i^1)_{m_i} (k_i^2)_{m_i} \\
&\qquad \cdot \ex\left[(W_{t_1}^{\diamond l_1-m_1} \diamond \cdots \diamond W_{t_K}^{\diamond l_K-m_K}) ( W_{t_1}^{\diamond l_1-m_1} \diamond \cdots \diamond W_{t_K}^{\diamond l_K-m_K})\right]\\
&\leq \sum_{\substack{(k_1^1,\ldots), (k_1^2, \ldots) \in \N^K\\ k_i^1, k_i^2 \geq m_i}} \frac{C^{k_1^1+k_1^2}t_1^{k_1^1+k_1^2-2m_1}}{(k_1^1-m_1)!(k_1^2-m_1)!} \cdots \frac{C^{k_K^1+k_K^2}t_K^{k_K^1+k_K^2-2m_K}}{(k_K^1-m_K)!(k_K^2-m_K)!}\\
&=\left(\sum_{\substack{(k_1,\ldots)\in \N^K\\ k_i\geq m_i}}  \frac{C^{k_1}t_1^{k_1-m_1}}{(k_1-m_1)!} \cdots \frac{C^{k_K}t_K^{k_K-m_K}}{(k_K-m_K)!}\right)^2 < \infty,
\end{align*}
i.e. the $L^2$ norms of all derivatives of $F^{\diamond}(W_{t_1}, \ldots, W_{t_K})$ are absolutely convergent. Since the multivariate Gaussian distribution $(W_{t_1}, \ldots, W_{t_K})$ has a continuous density on $\R^K$, we obtain the absolute convergence of the analytic representation $f^{\diamond}$ and this completes the assertion (ii). 
\end{proof}

%\begin{remark}
%(i) \ We mention that not every $f \in C^{\infty}([0,1]^K\times \R^{K},\R)$ gives an analytic representation of a Wick-analytic functional $F^{\diamond}(W_{t_1}, \ldots, W_{t_K})$, see Theorem \ref{thm_exact_simulation_is_Wick_analytic} below.
%
%(ii) \ The statements in Proposition \ref{prop_Wick_product_Gaussian_analytic} follow analogously for general Wick-analytic functionals \eqref{eq_Wick_analytic_functional}.
%
%(iii) \ The class of Wick-analytical functionals \eqref{eq_Wick_analytic_functional} is not a subclass or a superclass of  the class of smooth random variables in Malliavin calculus. However it is contained in the domain $\D^{1,2}$ of the Malliavin derivative, see \cite[1.2]{Nualart}.
%\end{remark}

\begin{remark}\label{remark_f_diamond_smoothness}
Suppose $f \in C^{1;2,\ldots, 2}([0,1]\times \R^{K})$ and fixed  $\tau_2, \ldots, \tau_K \in [0,1]$. Then, by Proposition \ref{prop_Wick_analytic_functionals_nice_properties} (ii), the representation 
\[
f(t,W_{t}, W_{\tau_2},\ldots, W_{\tau_K}) = f^{\diamond}(t, \tau_2, \ldots, \tau_K,W_{t}, W_{\tau_2},\ldots, W_{\tau_K}) 
\]
satisfies in particular $f^{\diamond} \in C^{1,\ldots, 1; 2,\ldots, 2}([0,1]^K\times \R^{K})$.
\end{remark}

A direct consequence of Proposition \ref{prop_Wickexp_properties} (iv) and the Wiener chaos decomposition is the totality of Wick-analytic functionals in $L^2(\Omega)$. In particular, we have:

\begin{proposition}\label{prop_square_integrable_Wick_analytic_approx}
Suppose $f:\R^K \rightarrow \R$ with $f(W_{t_1}, \ldots, W_{t_K}) \in L^2(\Omega)$ for some fixed $t_i \in [0,1]$. Then there exists a sequence of Wick-analytic functionals $F^{\diamond}_m$ with 
\[
F^{\diamond}_m(W_{t_1}, \ldots, W_{t_K}) \rightarrow  f(W_{t_1}, \ldots, W_{t_K}) \quad in \ L^2(\Omega) \ as \ m\rightarrow \infty.
\]
\end{proposition}

The enormous advantage of the Wick product is that it is preserved by conditional expectation. This is a direct consequence of \cite[Corollary 9.4]{Janson} or \cite[Lemma 6.20]{Di_Nunno_und_so}:

\begin{proposition}\label{prop_Wick_conditional_expectation}
For $X,Y, X\diamond Y \in L^2(\Omega)$ and the sub-$\sigma$-field $\mathcal{G} \subseteq \mathcal{F}$:
$$
\ex[X\diamond Y|\mathcal{G}] = \ex[X|\mathcal{G}]\diamond \ex[Y|\mathcal{G}].
$$ 
\end{proposition}

We consider the equidistant discretization of the underlying Brownian motion
$$
\P_n:= \{W_{\frac{1}{n}}, W_{\frac{2}{n}}, \ldots, W_1\}
$$
and the linear interpolation with respect to this discretization, i.e.
$$
W_t^{\lin} := W_{i/n} + n(t-i/n)(W_{(i+1)/n}-W_{i/n}),
$$
if  $i/n \leq t \leq (i+1)/n$. We clearly have $W_t^{\lin} = \ex[W_t|\P_n]$. Due to Proposition \ref{prop_S_transform_properties} (iii) and Proposition \ref{prop_Wick_conditional_expectation}, we have for Wick analytic functionals,
\begin{equation}\label{eq_cond_exp_simple_Skorohod_integral}
\ex[\int_{0}^{1}F^{\diamond}(W_{t_1}, \ldots, W_{t_K}) dW_s|\P_n] = F^{\diamond}(W_{t_1}^{\lin}, \ldots, W_{t_K}^{\lin}) \diamond W_1,
\end{equation}
if the Skorohod integral and the right hand side exist in $L^2(\Omega)$. Moreover, we have

\begin{theorem}\label{thm_cond_exp_Skorohod_as_Wick_product}
Suppose $f:\R^K \rightarrow \R$ with $f(W_{t_1}, \ldots, W_{t_K}) \in L^2(\Omega)$ for some fixed $t_i \in [0,1]$ and a (Skorohod integrable) process  $u\in L^2(\Omega \times [0,1])$. Then
\[
\ex[\int_{0}^{1}f(W_{t_1}, \ldots, W_{t_K}) \diamond u_s dW_s|\P_n] = f(W_{t_1}^{\lin}, \ldots, W_{t_K}^{\lin}) \diamond \ex[\int_{0}^{1} u_s dW_s|\P_n], 
\]
if both sides exist in $L^2(\Omega)$. In particular, \eqref{eq_cond_exp_simple_Skorohod_integral} is true for random variables $f(W_{t_1}, \ldots, W_{t_K}) \in L^2(\Omega)$ as well.
\end{theorem}

\begin{proof}
For Wick-analytic $F^{\diamond}(W_{t_1}, \ldots, W_{t_K})$, thanks to Proposition \ref{prop_S_transform_properties} (iii) and Proposition \ref{prop_Wick_conditional_expectation}, we have
\[
\ex[\int_{0}^{1}F^{\diamond}(W_{t_1}, \ldots, W_{t_K}) \diamond u_s dW_s|\P_n] = \ex[F^{\diamond}(W_{t_1}, \ldots, W_{t_K})|\P_n] \diamond \ex[\int_{0}^{1} u_s dW_s|\P_n].
\]
Otherwise, let $(F^{\diamond}_m(W_{t_1}, \ldots, W_{t_K}))_{m \geq 1}$ be the sequence of Wick-analytic functionals from Proposition \ref{prop_square_integrable_Wick_analytic_approx}. Due to Proposition \ref{prop_Wick_conditional_expectation}, we have
\begin{equation}\label{eq_thm_cond_exp_Skorohod_as_Wick_product1}
 \ex[F^{\diamond}_m(W_{t_1}, \ldots, W_{t_K})|\P_n]=F^{\diamond}_m(W_{t_1}^{\lin}, \ldots, W_{t_K}^{\lin}).
\end{equation}

For all $s,t \in [0,1]$, we observe
\begin{align}
\ex[W_s^{\lin} W_t^{\lin}] 
&=(s\wedge t)\eins_{\{\lfloor ns\rfloor \neq \lfloor nt\rfloor\}} + \left(\lfloor nt\rfloor/n + n(t-\lfloor nt\rfloor/n)(s-\lfloor nt\rfloor/n)\right)\eins_{\{\lfloor ns\rfloor=\lfloor nt\rfloor\}}
\label{eq_cov_W_lin}
\end{align}
and thus in particular
\begin{equation*}
\ex[W_s^{\lin} W_t^{\lin}] \leq s\wedge t = \ex[W_s W_t]. 
\end{equation*}
This gives that the measure of the multivariate Gaussian distributed $(W_{t_1}^{\lin}, \ldots, W_{t_K}^{\lin})$ is absolutely continuous with respect to the measure of $(W_{t_1}, \ldots, W_{t_K})$ with a bounded Radon-Nikodym density. Hence, the convergence in Proposition \ref{prop_square_integrable_Wick_analytic_approx} implies
$F^{\diamond}_m(W_{t_1}^{\lin}, \ldots, W_{t_K}^{\lin}) \rightarrow f(W_{t_1}^{\lin}, \ldots, W_{t_K}^{\lin})$ in $L^2(\Omega)$. Thus, by \eqref{eq_thm_cond_exp_Skorohod_as_Wick_product1} and the Cauchy-Schwarz inequality, we obtain for all $h \in L^2([0,1])$,
\begin{align*}
\displaybreak[0]
&\left|S\left (\ex[f(W_{t_1}, \ldots, W_{t_K})|\P_n] - f(W_{t_1}^{\lin}, \ldots, W_{t_K}^{\lin})\right)(h)\right|\\ 
&\leq \left|(S\left (\ex[f(W_{t_1}, \ldots, W_{t_K})|\P_n] - \ex[F^{\diamond}_m(W_{t_1}, \ldots, W_{t_K})|\P_n]\right)(h)\right|\\
&\quad + \left|(S \left(F^{\diamond}_m(W_{t_1}^{\lin}, \ldots, W_{t_K}^{\lin}) - f(W_{t_1}^{\lin}, \ldots, W_{t_K}^{\lin})\right)(h)\right|\\
&\leq \ex[(f(W_{t_1}, \ldots, W_{t_K})-F^{\diamond}_m(W_{t_1}, \ldots, W_{t_K})|\P_n])^2]^{1/2}\ex[(\exp^{\diamond}(I(h)))^2]^{1/2}\\
&\quad + \ex[(F^{\diamond}_m(W_{t_1}^{\lin}, \ldots, W_{t_K}^{\lin})-f(W_{t_1}^{\lin}, \ldots, W_{t_K}^{\lin}))^2]^{1/2}\ex[(\exp^{\diamond}(I(h)))^2]^{1/2} \rightarrow 0,
\end{align*}
as $m$ tends to infinity. Hence, by the injectivity of the S-transform, we conclude 
\[
\ex[f(W_{t_1}, \ldots, W_{t_K})|\P_n] \diamond \ex[\int_{0}^{1} u_s dW_s|\P_n] =  f(W_{t_1}^{\lin}, \ldots, W_{t_K}^{\lin}) \diamond \ex[\int_{0}^{1} u_s dW_s|\P_n].
\]
\end{proof}

%\begin{example}
%By the integration by parts formula (see e.g. \cite[Theorem 6.15]{Di_Nunno_und_so}) and the Malliavin derivative $D_s |W_{1/2}| = sign(W_{1/2})\eins_{[0,1/2]}(s)$, we obtain the Skorohod integral 
%\[
%\int_{0}^{1} |W_{1/2}| dW_s = |W_{1/2}| \diamond W_1= |W_{1/2}| W_1 - \frac{1}{2} sign(W_{1/2}). 
%\]
%Analogously, due to Theorem \ref{thm_cond_exp_Skorohod_as_Wick_product} and  
%\[
%D_s |W_{1/2}^{\lin}| = sign(W_{1/2}^{\lin})\left(\eins_{[0,\lfloor n/2\rfloor/n]}(s)+(n/2-\lfloor n/2\rfloor)\eins_{[\lfloor n/2\rfloor/n, \lfloor n/2\rfloor+1)/n]}(s)\right),
%\]
%the optimal approximation is given by
%\[
%\ex[\int_{0}^{1} |W_{1/2}| dW_s|\P_n] =  |W_{1/2}^{\lin}| \diamond W_1= |W_{1/2}^{\lin}| W_1 - \frac{1}{2} sign(W_{1/2}^{\lin}).  
%\]
%One can check that $|W_{1/2}|$ is not a Wick-analytic functional. This example of a nonsmooth integrand $f(x)=|x|$ is discussed in an accompanying article in more detail.
%\end{example}

We will need the following It\^o formula for Skorohod integrals. The proof is paradigmatic for further results, moreover it explains the appearance and importance of $f^{\diamond}$.

\begin{theorem}\label{thm_Ito_formula_Skorohod}
Suppose $K \in \N$, $f \in C^{1;2,\ldots, 2}([0,1]\times \R^{K})$ and fixed  $\tau_2, \ldots, \tau_K \in [0,1]$. Then $\left(\dfrac{\partial}{\partial x_1}f(t,W_{t}, W_{\tau_2},\ldots, W_{\tau_K})\right)_{t \in [0,1]}$ is Skorohod integrable and 
\begin{align}
&f(1,W_{1}, W_{\tau_2},\ldots, W_{\tau_K}) - f(0,W_{0}, W_{\tau_2},\ldots, W_{\tau_K})\nonumber\\ &=\int_{0}^{1}\dfrac{\partial}{\partial x_1} f^{\diamond}(t, \tau_2\ldots, \tau_K,W_{t}, \ldots, W_{\tau_K}) dW_{t} + \int_{0}^{1} \mathcal{L} f^{\diamond}(t, \tau_2\ldots, \tau_K,W_{t}, \ldots, W_{\tau_K})dt,\label{eq_thm_Ito_formula_Skorohod}
\end{align}
with $f^{\diamond} \in C^{1,\ldots, 1; 2,\ldots, 2}([0,1]^K\times \R^{K})$ from Remark \ref{remark_Wiener_chaos_by_Wick_products} and Remark \ref{remark_f_diamond_smoothness}
%\begin{equation} f^{\diamond} (\tau_1, \tau_2\ldots, \tau_K, x_1, \ldots, x_K) = \sum_{l_1, \ldots, l_K \geq 0} a_{l_1, \ldots, l_K}(\tau_1) h_{l_1, \ldots, l_K}(\tau_1, \ldots, \tau_K,x_1, \ldots, x_K) \end{equation}   where the $ a_{l_1, \ldots, l_K}:[0,1] \rightarrow \mathbb{R}$ arise from the chaos decomposition \eqref{eq_Wiener_chaos_by_Wick_products}, i.e.\begin{equation} f(t,W_{t_1}, \ldots, W_{t_K}) = \sum_{l_1, \ldots, l_K\geq 0} a_{l_1, \ldots, l_K}(t) W_{t_1}^{\diamond l_1} \diamond \cdots \diamond W_{t_K}^{\diamond l_K}, \qquad t \in [0,1], \end{equation}and the functions $h_{l_1, \ldots, l_K}:[0,1]^K \times \mathbb{R}^K \rightarrow \mathbb{R}$ are given by Proposition \ref{prop_Wick_product_Gaussian_analytic}, i.e.\begin{equation} W_{\tau_1}^{\diamond l_1} \diamond \cdots \diamond W_{\tau_K}^{\diamond l_K} =  h_{l_1, \ldots, l_K}(\tau_1, \ldots, \tau_K,W_{\tau_1}, \ldots, W_{\tau_K}). \end{equation}
%{\tt Ist die Darstellung von $f^{\diamond}$ richtig ? ...}
and the differential operator $\mathcal{L}$ is defined as
\begin{equation}\label{eq_L_operator}
\mathcal{L}:= \left(\sum_{1\leq k \leq K}\dfrac{\partial}{\partial t_k} + \frac{1}{2} \sum_{1\leq k,l \leq K} \dfrac{\partial^2}{\partial x_k \partial x_l}\right).  
\end{equation}
\end{theorem}

%{\tt ... Wie steht die Glattheit von $f$ mit der von $f^{\diamond}$ in Beziehung ...???}
\begin{proof} The assumptions in the theorem ensure the existence and linearity of the right hand side in \eqref{eq_thm_Ito_formula_Skorohod}. Thanks to the continuity assumption and Remark \ref{remark_Wiener_chaos_by_Wick_products}, we have a  representation
\begin{equation}\label{eq_thm_Ito_formula_Skorohod1}
f(t,W_{t}, W_{\tau_2}, \ldots, W_{\tau_K}) = \sum_{l_1, \ldots, l_K\geq 0} a_{l_1, \ldots, l_K}(t) W_{t}^{\diamond l_1} \diamond W_{\tau_2}^{\diamond l_2} \diamond \cdots \diamond W_{\tau_K}^{\diamond l_K},  
\end{equation}
where $a_{l_1, \ldots, l_K}(\cdot)\in C^1([0,1])$ for all $l_1, \ldots, l_K \geq 0$. 
Due to linearity and orthogonality, it suffices to consider the finite chaoses $a_{l_1, \ldots, l_K}(t) W_{t}^{\diamond l_1} \diamond W_{\tau_2}^{\diamond l_2} \diamond \cdots \diamond W_{\tau_K}^{\diamond l_K}$ in \eqref{eq_thm_Ito_formula_Skorohod1}. By \eqref{eq_Wick_product_by_Wick_exp} the terms $W_{t}^{\diamond l_1} \diamond W_{\tau_2}^{\diamond l_2} \diamond \cdots \diamond W_{\tau_K}^{\diamond l_K}$ are just polynomials in $W_{t}, W_{\tau_2}, \ldots, W_{\tau_K}$ and $t, \tau_2,\ldots, \tau_K$. For different $t_1, \ldots, t_K$ and arbitrary $a_1,\ldots, a_K \in \R$, we clearly have
\[
\mathcal{L} \exp\left(\sum_{i} a_i x_i - \frac{1}{2}\sum_{i,j} a_i a_j (t_i \wedge t_j)\right) = 0.
\]
Hence,
\[
\mathcal{L} \exp^{\diamond} \left(\sum a_i W_{t_i}\right) =\mathcal{L} \exp\left(\sum_{i} a_i W_{t_i} - \frac{1}{2}\|\sum a_i \eins_{[0,t_1]}\|^2\right)=0 
\]
on the right hand side in \eqref{eq_Wick_product_by_Wick_exp} and this yields
\begin{equation}\label{eq_thm_Ito_formula_Skorohod2}
\mathcal{L} \left(W_{t}^{\diamond l_1} \diamond W_{\tau_2}^{\diamond l_2} \diamond \cdots \diamond W_{\tau_K}^{\diamond l_K}\right) =0
\end{equation}
via Proposition \ref{prop_Wick_product_by_Wick_exp}.
Moreover, for a function $f \in C^1([0,1])$ and all integers $k \geq 1$, we have the integration by parts formula
\begin{equation}\label{eq_stochastic_int_by_parts_Wick_powers}
f(1) W_1^{\diamond k} - f(0)W_0^{\diamond k} = \int_{0}^{1} f(s) kW_s^{\diamond k-1} dW_s + \int_{0}^{1} f'(s) W_s^{\diamond k} ds, 
\end{equation}
which follows by the ordinary It\^o formula or by S-transform as follows: By the ordinary integration by parts formula, we have for all $h \in L^2([0,1])$,
\begin{align*}
&f(1) \left(\int_0^1 h(u) du\right)^k - f(0) \left(\int_{0}^0 h(u) du\right)^{k}\\
&= \int_{0}^1 f(s) k  \left(\int_0^s h(u) du\right)^{k-1} h(s) ds + \int_{0}^1 f'(s) \left(\int_0^s h(u) du\right)^{k} ds.
\end{align*}
This yields the equality of the S-transform $(S \cdot)(h)$ applied on both sides of \eqref{eq_stochastic_int_by_parts_Wick_powers}. Thus, by the injectivity of the S-transform, we obtain \eqref{eq_stochastic_int_by_parts_Wick_powers}. Then, due to Proposition \ref{prop_S_transform_properties} (iii), (iv), the further Wick products in the chaoses applied on \eqref{eq_stochastic_int_by_parts_Wick_powers} carry over and we obtain
\begin{align}
&a_{l_1, \ldots, l_K}(1) W_1^{\diamond l_1} \diamond \cdots \diamond W_{\tau_K}^{\diamond l_K} - a_{l_1, \ldots, l_K}(0) W_0^{\diamond l_1} \diamond \cdots \diamond W_{\tau_K}^{\diamond l_K}\nonumber\\
&= \int_{0}^{1} l_1a_{l_1, \ldots, l_K}(t) W_{t}^{\diamond l_1-1} \diamond \cdots \diamond W_{\tau_K}^{\diamond l_K} dW_{t}  + \int_{0}^{1} \left(\dfrac{\partial}{\partial t}a_{l_1, \ldots, l_K}(t)\right) W_{t}^{\diamond l_1} \diamond \cdots \diamond W_{\tau_K}^{\diamond l_K} ds.\label{eq_thm_Ito_formula_Skorohod3}
\end{align}
Alternatively, analogously to \eqref{eq_stochastic_int_by_parts_Wick_powers}, one can check \eqref{eq_thm_Ito_formula_Skorohod3} directly by S-transforms. Thanks to \eqref{eq_Wick_product_by_Wick_exp} and \eqref{eq_thm_Ito_formula_Skorohod2}, we have
\begin{align}
l_1a_{l_1, \ldots, l_K}(t) W_{t}^{\diamond l_1-1} \diamond \cdots \diamond W_{\tau_K}^{\diamond l_K}  &= \dfrac{\partial}{\partial x_1}a_{l_1, \ldots, l_K}(t) W_{t}^{\diamond l_1} \diamond \cdots \diamond W_{\tau_K}^{\diamond l_K},\nonumber\\
\left(\dfrac{\partial}{\partial t}a_{l_1, \ldots, l_K}(t)\right) W_{t}^{\diamond l_1} \diamond \cdots \diamond W_{\tau_K}^{\diamond l_K} &= \mathcal{L} \left(a_{l_1, \ldots, l_K}(t) W_{t}^{\diamond l_1} \diamond \cdots \diamond W_{\tau_K}^{\diamond l_K}\right).\label{eq_thm_Ito_formula_Skorohod4}
\end{align}
Thus, via \eqref{eq_thm_Ito_formula_Skorohod3}, we conclude the asserted It\^o formula for all finite chaoses in \eqref{eq_thm_Ito_formula_Skorohod1}. The strong convergence of the Wiener chaos decomposition of the left hand side and of the pathwise integral $\int_{0}^{1} \mathcal{L} f^{\diamond}(t,\ldots, W_{\tau_K})dt$ in \eqref{eq_thm_Ito_formula_Skorohod} implies the existence and square integrability of the asserted Skorohod integral.
\end{proof}

\begin{remark}\label{remark_Wick_analytic_L_zero}
Thanks to \eqref{eq_thm_Ito_formula_Skorohod2} and the continuous density of $(W_{t_1}, W_{t_2},\ldots, W_{t_K})$ on $\R^K$, we conclude for the analytic representation $f^{\diamond} \in C^{\infty}([0,1]^K \times \R^K)$ of a Wick-analytic functional $F^{\diamond}(W_{t_1}, W_{t_2},\ldots, W_{t_K})$ via Proposition \ref{prop_Wick_product_Gaussian_analytic} that
\[
\mathcal{L} f^{\diamond} =0. 
\]
\end{remark}

\section{Exact simulation of Skorohod integrals}\label{Section_Exact_sim}

In this section we characterize, when the exact simulation of $I$ is possible, at least theoretically. We have the following conditions for exact simulation of Skorohod integrals and the connection to the Wick-analytic representation. 

\begin{theorem}\label{thm_exact_simulation_is_Wick_analytic}
Suppose $f \in C([0,1] \times \R^{K})$ and that
$\left(u_{t}= f(t ,W_{t}, W_{\tau_2},\ldots, W_{\tau_K})\right)_{t \in [0,1]}$ is Skorohod integrable for all fixed $\tau_2,\ldots, \tau_K \in [0,1]$. Then the following assertions are equivalent:
\begin{enumerate}
\item There exists a Wick-analytic representation of $u$ in sense of \eqref{eq_simple_Wick_analytic_functional}.
\item $f^{\diamond}\in C^{\infty}([0,1]^K\times \R^{K})$ and $\mathcal{L} f^{\diamond}=0$ on $[0,1]^K \times \R^{K}$ .
\item $f^{\diamond} \in C^{1,\ldots,1,2,\ldots, 2}([0,1]^K\times \R^{K})$ and $\mathcal{L} f^{\diamond} =0$ on $[0,1]^K \times \R^{K}$.
\item There exists a Borel function $h: \R^{K} \rightarrow \R$ with
 \[
h(W_1, W_{\tau_2},\ldots, W_{\tau_K}) = \int_{0}^{1}u_s dW_{s}.
 \]
\end{enumerate} 
%If any of the assertions in (i)-(iv) is true and $t_2,\ldots, t_K \in \frac{1}{n}\N \cap [0,1]$, then $e_n=0$. 

If any of the assertions in (i)-(iv) is true, then 
\[
\ex\left[\left(\int_{0}^{1}u_s dW_s-\ex[\int_{0}^{1}u_s dW_s|W_1, W_{\tau_2},\ldots, W_{\tau_K}]\right)^2\right]=0.
\] 
\end{theorem}

\begin{remark}
This result generalizes \cite[Theorem 3.1]{Przybylowicz} to nonadapted processes and the connection to the Wick-analytic representation simplifies the proof.
\end{remark}

As a direct Corollary we obtain: 
%{\tt ... Sind die Glattheitsannahmen ok ? ...}}

\begin{corollary}
 Suppose $f \in C^{1;2,\ldots, 2}([0,1]\times \R^{K})$ and denote 
$\left(u_{t}= f(t ,W_{t}, W_{\tau_2},\ldots, W_{\tau_K})\right)_{t \in [0,1]}$ (which is Skorohod integrable for all fixed $\tau_2,\ldots, \tau_K \in [0,1]$ by Theorem \ref{thm_Ito_formula_Skorohod}). Then we have
$$ I = \hat{I}_n \quad a.s. \qquad \Longleftrightarrow \qquad \int_{0}^{1} \ex[\mathcal{L} f^{\diamond}(s, \tau_2,\ldots, \tau_K,W_s,W_{\tau_2}, \ldots, W_{\tau_K})^2] ds =0.$$
\end{corollary}

\begin{proof}[Proof of Theorem \ref{thm_exact_simulation_is_Wick_analytic}]
Firstly we recall from Proposition \ref{prop_Wick_analytic_functionals_nice_properties} and Proposition \ref{prop_Wick_product_Gaussian_analytic} that for every Wick-analytic functional the (ordinary) analytic representation 
\[
F^{\diamond}(W_{t_1}, W_{t_2},\ldots, W_{t_K}) = f^{\diamond}(t_1,\ldots, t_K,W_{t_1}, W_{t_2},\ldots, W_{t_K})
\]

is in $C^{\infty}([0,1]^K \times \R^{K})$ and all derivatives converge absolutely (cf. the simple derivatives in \eqref{eq_thm_Ito_formula_Skorohod4} and Proposition \ref{prop_Wick_product_Gaussian_analytic}).

We prove $(i) \Rightarrow (ii) \Rightarrow (iii) \Rightarrow (i)$ and then $(i) \Rightarrow (iv) \Rightarrow (i)$.

$(i) \Rightarrow (ii)$: Suppose the representation 
\begin{equation}\label{eq_thm_exact_simulation_is_Wick_analytic1}
u_t = \sum_{k_1,\ldots, k_K \in \N} a_{k_1, \ldots, k_K} W_{t}^{\diamond k_1} \diamond W_{\tau_2}^{\diamond k_2} \diamond\cdots W_{\tau_K}^{\diamond k_K}. 
\end{equation}

Then, via Proposition \ref{prop_Wick_product_Gaussian_analytic} and Remark \ref{remark_Wick_analytic_L_zero}, we conclude assertion (ii). 

$(ii) \Rightarrow (iii)$: is clear.

$(iii) \Rightarrow (i)$: By the Wiener chaos decomposition, we have the representation \eqref{eq_thm_Ito_formula_Skorohod1} and it suffices to consider the individual chaoses
\[
a_{l_1, \ldots, l_K}(t) W_{t}^{\diamond k_1} \diamond W_{\tau_2}^{\diamond k_2} \diamond\cdots W_{\tau_K}^{\diamond k_K}.
\]
 in \eqref{eq_thm_Ito_formula_Skorohod1}. Hence, due to the assumption and the integration by parts formula \eqref{eq_thm_Ito_formula_Skorohod3}, we obtain that the coefficients $a_{l_1, \ldots, l_K}$ must be constant. Thus we obtain a representation \eqref{eq_thm_exact_simulation_is_Wick_analytic1}.

$(i) \Rightarrow (iv)$: Due to Proposition \ref{prop_S_transform_properties} (iii) and the definition of the Skorohod integral, for \eqref{eq_thm_exact_simulation_is_Wick_analytic1}, 
\[
\int_{0}^{1}u_t dW_{t} = \sum_{k_1,\ldots, k_K \in \N} a_{k_1, \ldots, k_K} \frac{1}{k_1+1}W_{1}^{\diamond k_1+1} \diamond \cdots W_{\tau_K}^{\diamond k_K},
\]
which is a Wick-analytic functional and exhibits an analytic representation $h(W_{1}, W_{\tau_2},\ldots, W_{\tau_K})$ for some $h \in C^{\infty}(\R^K)$ (and fixed $\tau_2,\ldots, \tau_K$) via Proposition \ref{prop_Wick_product_Gaussian_analytic}.

$(iv) \Rightarrow (i)$: We define the function
\[
g(s_1, \ldots, s_K, x_1,\ldots, x_K) := \ex[h(x_1 + W_{1-s_1}, x_2 + W_{\tau_2-s_2}, \ldots, x_K+W_{\tau_K-s_K})],
\]
which is well-defined by Skorohod integrability of $u$ and $\ex[h(W_{1}, W_{\tau_2},\ldots, W_{\tau_K})^2] <\infty$. Then, analogously to the case $K=1$ (see e.g. \cite[Section 4.3]{Karatzas_Shreve}), we obtain by computations reduced on the multivariate heat kernel that $g \in C^{\infty}([0,1]\times [0,\tau_2] \times \cdots \times [0,\tau_K] \times \R^K)$ and $\mathcal{L} g = 0$. Hence, by the zero mean property of Skorohod integrals and the It\^o formula in Theorem \ref{thm_Ito_formula_Skorohod}, we have
\begin{align*}
&h(W_1, W_{\tau_2},\ldots, W_{\tau_K})\\
&=\left.\ex[h(x_1 + W_{0}, x_2 + W_{0}, \ldots, x_K+W_{0})]\right|_{x_1=W_1, x_2=W_{\tau_2}, \ldots, x_K=W_{\tau_K}}\\
&\quad -\left.\ex[h(x_1+W_{1}, x_2 + W_{0}, \ldots, x_K+W_{0})]\right|_{x_1=W_0, x_2=W_{\tau_2},\ldots, x_K=W_{\tau_K}}\\
&= g(1,\tau_2,\ldots, \tau_K, W_1,W_{\tau_2},\ldots, W_{\tau_K}) - g(0,\tau_2,\ldots, \tau_K, W_0,W_{\tau_2},\ldots, W_{\tau_K})\\
&=\int_{0}^{1}\dfrac{\partial}{\partial x_1}g(t,\tau_2,\ldots, \tau_K, W_t,W_{\tau_2},\ldots, W_{\tau_K})dW_{t}.
\end{align*}
Thus we obtain
\[
f^{\diamond}(t,\tau_2,\ldots, \tau_K, W_t,W_{\tau_2},\ldots, W_{\tau_K})=\dfrac{\partial}{\partial x_1}g(t,\tau_2,\ldots, \tau_K, W_t,W_{\tau_2},\ldots, W_{\tau_K}) \ (a.s.) .
\]
Via $g \in C^{\infty}([0,1]^K \times \R^K)$ and $\mathcal{L} g = 0$ and the already proved part $(iii) \Rightarrow (i)$, we get a Wick-analytic representation of $g(t,\tau_2,\ldots, \tau_K, W_t,W_{\tau_2},\ldots, W_{\tau_K})$. Thus, by Proposition \ref{prop_Wick_product_Gaussian_analytic} (iii), we conclude a Wick-analytic representation of $f^{\diamond}(t,\tau_2,\ldots, \tau_K, W_t,W_{\tau_2},\ldots, W_{\tau_K})$.

The statement on exact approximation is a direct consequence of (iv).
\end{proof}

\begin{remark}
Whether the integral $I$ can be exactly simulated or not, is determined by the pathwise defect 
\[
\int_{0}^{1} \mathcal{L} f^{\diamond}(t,\tau_2,\ldots, \tau_K, W_t,W_{\tau_2},\ldots, W_{\tau_K})dt \neq 0,
\]
due to Theorem \ref{thm_Ito_formula_Skorohod} and Theorem \ref{thm_exact_simulation_is_Wick_analytic}. The pathwise defect is again completely determined by the deviation of the nonconstant coefficients in \eqref{eq_thm_Ito_formula_Skorohod1} in contrast to the constant coefficients in a Wick-analytic functional.
\end{remark}

%\begin{example}
%Suppose $\tau_2,\ldots, \tau_K\in [0,1]$ are fixed. Then
%\[ 
%\ex\left[\left.\int_{0}^{1} e^{\sum_{i} W_{\tau_i} - 1/2\sum_{i,j} (\tau_i \wedge \tau_j)}dW_s\right|\{W_1, W_{\tau_2},\ldots, W_{\tau_K}\}\right]= \int_{0}^{1} e^{\sum_{i} W_{\tau_i} - 1/2\sum_{i,j} (\tau_i \wedge \tau_j)}dW_s
%\]
%whereas
%\[
%\ex[\int_{0}^{1} \exp(\sum_{i} W_{\tau_i})dW_s|\{W_1, W_{\tau_2},\ldots, W_{\tau_K}\}]\neq  \int_{0}^{1} \exp(\sum_{i} W_{\tau_i})dW_s, \qquad
%\ex\left[\left.\int_{0}^{1} |\sum_{i} W_{\tau_i}|^2dW_s\right|\{W_1, W_{\tau_2},\ldots, W_{\tau_K}\}\right] \neq \int_{0}^{1} |^2\sum_{i} W_{\tau_i}|dW_s.
%\] 
%\end{example}

\section{Optimal approximation of Skorohod integrals}\label{Section_Optimal_approx_Skorohod_integrals}

We restrict ourselves to the simple case of an equidistant discretization  plus the knowledge of the non-adapted part, i.e.
\[
\P_n:= \{W_{\frac{1}{n}}, W_{\frac{2}{n}}, \ldots, W_1, W_{\tau_2}, \ldots, W_{\tau_K} \}.
\]
We denote the Skorohod integral as
\[
I := \int_{0}^{1}u_s dW_s 
\]
and  its optimal $L^2$-approximation as
\[
\hat{I}^n:= \ex[I|\P_n] = \ex[\int_{0}^{1}u_s dW_s|\P_n]. 
\]
The mean square error is denoted by
\[
e_n := \ex[(I-\hat{I}^n)^2]^{1/2}. 
\]
In contrast to It\^o integrals, the optimal approximation of Skorohod integrals of the type 
\begin{equation}\label{eq_Skorohod_integral_opt_approx}
I = \int_{0}^{1} f(t, \tau_2,\ldots, \tau_K,W_{t}, W_{\tau_2}, \ldots, W_{\tau_K}) dW_{t}, 
\end{equation}
depends on the information $W_{\tau_2}, \ldots, W_{\tau_K}$, i.e. the nonadapted fixed timepoints $\tau_2,\ldots, \tau_K$. So, for the sake of notational simplicity, we here use the representation $f=f^{\diamond}$.  For the behaviour of the minimal error we obtain 
\[
e_n \approx C \cdot n^{-1},
\]
which generalizes the It\^o case and leads to a deeper understanding of the constant in the optimal convergence rate. %This first result is based on a finite set $t_2,\ldots, t_K \in \frac{1}{m}\N$, i.e. $W_{t_2},\ldots, W_{t_K} \in \P_{m n}$ for all $n \in \N$, which is, roughly saying, the perfect knowledge of the nonadapted terms in the integrand $f(t_1, \ldots, t_K,W_{t_1},\ldots, W_{t_K})$. 
%Beyond this property we have a rather irregular and unknown integrand. Due to the Weyl equidistribution theorem, we study the best case scenario for optimal approximation with $\{W_{t_2},\ldots, W_{t_K}\} \cap \P_{n} \neq \emptyset$ for all $n \in \N$. This leads to an optimal approximation \[ e_n \sim c \cdot n^{-1/2}.\]
We determine this optimal approximation rate and the constant for a large class of integrands. 
%As a byproduct, we conclude a new result on optimal approximation of It\^o integrals for a class of discontinuous integrands. 

The optimal approximation of It\^o integrals uses the following Lipschitz and linear growth conditions on $f: [0,1] \times \R \rightarrow \R$:
\begin{itemize}
 \item[(L\_1)] There exists a constant $c>0$, such that for all $t \in [0,1]$, $x,y\in\R$,
 \[
|f(t,x)-f(t,y)| \leq c|x-y|.
 \]
\item[(LLG)] There exists a constant $c>0$, such that for all $s,t \in [0,1]$, $x\in\R$,
 \[
|f(s,x)-f(t,x)| \leq c(1+|x|)|s-t|.
 \]
\end{itemize}
Then, the class $F_1$ is defined by $f \in C^{1,2}([0,1] \times \R)$ such that the partial derivative $f^{(1,0)}$ satisfies (LLG) and $f^{(0,1)}$, $f^{(0,2)}$ satisfy (L\_1) and (LLG). For $f \in F_1$, the It\^o integral $I := \int_{0}^{1}f(s,W_s) dW_s$ and its optimal approximation $\tilde{I}^n:= \ex[I| W_{\frac{1}{n}}, W_{\frac{2}{n}}, \ldots, W_1]$, the convergence 
\begin{equation}\label{eq_optimal_approx_simple_Ito_rate}
\lim\limits_{n \rightarrow \infty} n \cdot 
\ex[(I- \tilde{I}^n  )^2]^{1/2}
= \frac{1}{\sqrt{12}} \left(\int_{0}^{1} \ex[\mathcal{L} f(s,W_s)^2] ds\right)^{1/2},
\end{equation}
is established (cf. \cite{Przybylowicz, Mueller_Gronbach}).

We assume the integral in \eqref{eq_Skorohod_integral_opt_approx} to exist and consider weaker regularity assumptions. The assumptions are
a Lipschitz continuity and a H\"older growth condition on $\mathcal{L} f^{\diamond}$ from Remark \ref{remark_Wiener_chaos_by_Wick_products} and Theorem \ref{thm_Ito_formula_Skorohod}:
\begin{itemize}
 \item[(L)] There exists a constant $c>0$, such that for all $t= (t_1,t_2 \ldots, t_K) \in [0,1]^K$, $x,y \in\R^K$,
 \begin{align*}
&\left| \mathcal{L} f^{\diamond}(t, x)-\mathcal{L} f^{\diamond}(t, y)\right|\leq c|x-y|_{\R^K}.
\end{align*}
 \item[(HG)] There exist constants $c>0$, $\ve>0$, such that for all $u,v \in [0,1]^K$, $x \in\R^K$,
 \begin{align*}
&\left| \mathcal{L} f^{\diamond}(u,x)-\mathcal{L} f^{\diamond}(v,x)\right|\leq c(1+|x|_{\R^K})|u-v|^{\ve},  
 \end{align*}
\end{itemize}

Similarly to the It\^o integral in \eqref{eq_optimal_approx_simple_Ito_rate}, but under weaker assumptions on regularity, we have the following main result on optimal approximation of Skorohod integrals:

\begin{theorem}\label{thm_optimal_approx_simple_Skorohod}
Suppose $K\in \N$, $ f=f^{\diamond}\in C^{1,\ldots,1,2,\ldots, 2}([0,1]^K\times \R^{K})$ with (L) and (HG) and let 
\[
C:= \frac{1}{\sqrt{12}} \left(\int_{0}^{1} \ex[\mathcal{L} f^{\diamond} (s, \tau_2,\ldots, \tau_K,W_s,W_{\tau_2}, \ldots, W_{\tau_K})^2] ds\right)^{1/2}. 
\]
%For $t_2, \ldots, t_K \in \frac{1}{m}\N$ fixed, it is \begin{equation}\label{eq_thm_optimal_approx_simple_Skorohod_rate}\lim\limits_{n \rightarrow \infty\ (n \in m\N)} n \cdot e_n = c_1.\end{equation}
%For $t_2, \ldots, t_K \in [0,1]$ fixed, 
Then
\begin{equation*}
\lim\limits_{n \rightarrow \infty} n \cdot \ex[(I-\ex[I|\P_n ])^2]^{1/2} = C.
\end{equation*}
\end{theorem}

\begin{remark}\label{remark_main_thm}
$\left.\right.$
\begin{enumerate}
\item The assumptions of Theorem \ref{thm_optimal_approx_simple_Skorohod} together with Theorem \ref{thm_Ito_formula_Skorohod} ensure the existence of the Skorohod integral \eqref{eq_Skorohod_integral_opt_approx}.
 %\item For $\tau_2, \ldots, \tau_K \in \frac{1}{m}\N$ fixed, it is \begin{equation}\label{eq_thm_optimal_approx_simple_Skorohod_rate}\lim\limits_{n \rightarrow \infty\ (n \in m\N)} n \cdot e_n = c.
 %\end{equation}
 \item The Lipschitz and linear growth conditions on the integrand in \eqref{eq_optimal_approx_simple_Ito_rate} are crucial for the proofs in \cite{Przybylowicz} and \cite{Mueller_Gronbach} due to an reduction on Wagner-Platen schemes. In contrast, our proofs are based on the Wiener chaos decomposition of the integrand $f$, similarly to Theorem \ref{thm_Ito_formula_Skorohod}. As a byproduct, Theorem \ref{thm_optimal_approx_simple_Skorohod} generalizes \eqref{eq_optimal_approx_simple_Ito_rate} for It\^o integrals to integrands $f \in C^{1,2}([0,1] \times \R)$ with (HG) only.
 \item A further sufficient condition for the assertion is given in Remark \ref{remark2_main_thm} below.
\end{enumerate}
\end{remark}

Before we prove the main theorem, we firstly notice these helpful and essential computations on the Gaussian random variables involved:

\begin{remark}\label{remark_norm_computations} We denote the Brownian bridge as
\[
B^n_t := W_t - W_t^{\lin}.
\]
For all $s,t \in [0,1]$, $1 \leq i,j \leq n$, and via \eqref{eq_cov_W_lin}, we have 
\begin{align}
\displaybreak[0]
&\ex[W_s W_t^{\lin}] = \ex[W_s^{\lin} W_t^{\lin}] \nonumber\\ 
&\qquad = (s\wedge t)\eins_{\{\lfloor ns\rfloor \neq \lfloor nt\rfloor\}} + \left(\lfloor nt\rfloor/n + n(t-\lfloor nt\rfloor/n)(s-\lfloor nt\rfloor/n)\right)\eins_{\{\lfloor ns\rfloor=\lfloor nt\rfloor\}} \leq s\wedge t,\label{eq_covariances1}\\
&\ex[B^n_s B^n_t]  = \ex[B^n_s W_t]\nonumber\\ 
&\qquad = \left(s\wedge t - \lfloor nt\rfloor/n -(s-\lfloor nt\rfloor/n)(t-\lfloor nt\rfloor/n)n\right)\eins_{\{\lfloor ns\rfloor=\lfloor nt\rfloor\}} \leq (4n)^{-1},\label{eq_covariances3}\\
&\ex[B^n_s W_t^{\lin}] =0,\label{eq_covariances4}\\
&\int_{(i-1)/n}^{i/n}\int_{(j-1)/n}^{i/n} \ex[B^n_s B^n_t] ds dt = \eins_{\{i=j\}} \frac{1}{12}n^{-3}.\label{eq_covariances5}
\end{align}
\end{remark}

\begin{proof}[Proof of Theorem \ref{thm_optimal_approx_simple_Skorohod}]
We assume $\{\tau_2,\ldots, \tau_K\} \in \frac{1}{n}\N$ and use $(t_2,\ldots, t_K)=(\tau_2,\ldots, \tau_K)$ for notational shortance. The proof of the general assertion is a straightforward modification. We make use of the shorthand notations
\begin{align}
\bar{l} := (l_1,\ldots, l_K) \in \N^K,\quad |\bar{l}| :=\sum l_i,\quad  a'_{\bar{l}}(s) := \dfrac{\partial}{\partial s}a_{l_1,\ldots, l_K}(s).\label{eq_shorthand_l} 
\end{align}

The proof is divided into three steps. 

In the the first step, the Lipschitz continuity and H\"older growth condition are applied to establish some upper bounds for the terms involved in the following proof.

The second step is devoted to a simplification of the mean squared error. It will be shown that the computation of 
$$
\ex\left[\left(\int_{0}^{1}u^M_{s} dW_s - \ex[\int_{0}^{1}u^M_{s} dW_s|\P_n]\right)^2\right]
$$
can be reduced to the computation of
$$
\sum_{M\geq 0}\ex\left[\left(\sum_{i=1}^{n}\int_{(i-1)/n}^{i/n} B^n_s \diamond \sum_{|\bar{l}| = M} a'_{\bar{l}}(s) W_s^{\diamond l_1} \diamond \ldots \diamond W_{t_K}^{\diamond l_K} ds\right)^2\right],
$$
since the difference between both is of  order $\mathcal{O}(n^{-3})$.

Finally, in the third step, the asserted constant
$$
\frac{1}{\sqrt{12}} \left(\int_{0}^{1} \ex[\mathcal{L} f^{\diamond} (s, \tau_2,\ldots, \tau_K,W_s,W_{\tau_2}, \ldots, W_{\tau_K})^2] ds\right)^{1/2}
$$
is identified in the simplified mean squared error and the analysis is completed by the upper bounds from Step 1.

\textit{Step 1}: 

We firstly observe some upper bounds by the regularity assumptions.  Thanks to \eqref{eq_thm_Ito_formula_Skorohod4}, the growth condition (HG) and the Cauchy-Schwarz inequality, for $u,v \in [0,1]$, we have
\begin{align}
\displaybreak[0]
&\sum_{M \geq 0} \ex\left[\left|\sum_{|\bar{l}| = M} \left(a'_{\bar{l}}(u) -  a'_{\bar{l}}(v)\right) W_{t_1}^{\diamond l_1} \diamond \ldots \diamond W_{t_K}^{\diamond l_K}\right|^2\right]\nonumber\\
&= \ex\left[|\sum_{M \geq 0} \sum_{|\bar{l}| = M} \left(a'_{\bar{l}}(u) -  a'_{\bar{l}}(v)\right) W_{t_1}^{\diamond l_1} \diamond \ldots \diamond W_{t_K}^{\diamond l_K}|^2\right]\nonumber\\ 
&=\ex[|\mathcal{L}f(u, t_2,\ldots, t_K,W_{t_1},W_{t_2}, \ldots, W_{t_K}) - \mathcal{L}f(v, t_2,\ldots, t_K,W_{t_1},W_{t_2}, \ldots, W_{t_K})|^2]\nonumber\\
&\leq 2\,c^2\,(1+\sum_{i=1}^{K} t_i)|u-v|^{2\ve} \leq 2\,c^2\,(K+1)|u-v|^{2\ve}.\label{eq_thm_optimal_approx_simple_Skorohod0}
\end{align}
Similarly, via the Lipschitz condition (L), we obtain
\begin{align}
\displaybreak[0]
&\sum_{M \geq 0} \ex\left[\left|\sum_{|\bar{l}| = M} a'_{\bar{l}}(s) \left(W_{u}^{\diamond l_1} -W_{v}^{\diamond l_1}\right)\diamond \ldots \diamond W_{t_K}^{\diamond l_K}\right|^2\right]\nonumber\\ 
&=\ex[|\mathcal{L}f(s, t_2,\ldots, t_K,W_{u},W_{t_2}, \ldots, W_{t_K}) - \mathcal{L}f(s, t_2,\ldots, t_K,W_{v},W_{t_2}, \ldots, W_{t_K})|^2]\nonumber\\
&\leq c^2\ex[|W_u-W_v|^2]=c^2|u-v|.\label{eq_thm_optimal_approx_simple_Skorohod00}
\end{align}
Hence, via \eqref{eq_thm_optimal_approx_simple_Skorohod0},  \eqref{eq_thm_optimal_approx_simple_Skorohod00} and $(a+b)^2 \leq 2(a^2+b^2)$, 
\begin{align}
\displaybreak[0]
&\sum_{M \geq 0} \ex\left[\left|\sum_{|\bar{l}| = M} \left(a'_{\bar{l}}(s)W_{u}^{\diamond l_1}\diamond \ldots \diamond W_{t_K}^{\diamond l_K} - a'_{\bar{l}}(t)W_{v}^{\diamond l_1}\diamond \ldots \diamond W_{t_K}^{\diamond l_K}\right)\right|^2\right]\nonumber\\ 
&=\ex[|\mathcal{L}f(s, t_2,\ldots, t_K,W_{u},W_{t_2}, \ldots, W_{t_K}) - \mathcal{L}f(t, t_2,\ldots, t_K,W_{v},W_{t_2}, \ldots, W_{t_K})|^2]\nonumber\\
&\leq 4c^2(K+1)|s-t|^{2\ve} + 2c^2|u-v|.\label{eq_thm_optimal_approx_simple_Skorohod000}
\end{align}
In particular, by \eqref{eq_thm_optimal_approx_simple_Skorohod00}-\eqref{eq_thm_optimal_approx_simple_Skorohod000}, $|bc-a^2|\leq |b-a||c-a| + |a|(|b-a|+|c-a|)$, the Cauchy-Schwarz inequality, for $s,t,u,v,w \in [0,1]$, $\max\{|s-t|, |u-w|,|v-w|\}\leq 1/n$, we conclude a constant $c' >0$, independently of $n$, such that
\begin{align}
\displaybreak[0]
&\sum_{M \geq 0} \left|\ex\left[\left(\sum_{|\bar{l}| = M} a'_{\bar{l}}(s)W_{u}^{\diamond l_1}\diamond \ldots \diamond W_{t_K}^{\diamond l_K}\right)\left(\sum_{|\bar{l}| = M} a'_{\bar{l}}(s)W_{v}^{\diamond l_1}\diamond \ldots \diamond W_{t_K}^{\diamond l_K}\right)\right] \right.\nonumber\\
&\qquad \left. - \ex\left[\left(\sum_{|\bar{l}| = M} a'_{\bar{l}}(t)W_{w}^{\diamond l_1}\diamond \ldots \diamond W_{t_K}^{\diamond l_K}\right)^2\right]\right|\nonumber\\
\displaybreak[0]
&=\ex[|\mathcal{L}f(s, t_2,\ldots,W_{u},\ldots) \mathcal{L}f(s, t_2,\ldots,W_{v},\ldots)  - \mathcal{L}f(t, t_2,\ldots,W_{w},\ldots)^2|]\nonumber\\
\displaybreak[0]
&\leq \left(4c^2(K+1)|s-t|^{2\ve} + 2c^2|u-w|\right)^{1/2}\left(4c^2(K+1)|s-t|^{2\ve} + 2c^2|v-w|\right)^{1/2}\nonumber\\ 
&\quad + 2c|w|^{1/2}\left(\left(4c^2(K+1)|s-t|^{2\ve} + 2c^2|u-w|\right)^{1/2} + \left(4c^2(K+1)|s-t|^{2\ve} + 2c^2|v-w|\right)^{1/2}\right)\nonumber\\
&\leq c'\, n^{-(\ve \wedge 1/2)}.\label{eq_thm_optimal_approx_simple_Skorohod0abc}
\end{align}
The Lipschitz continuity (L) implies that for fixed $t := (t_1,\ldots, t_K) \in [0,1]^{K}$, all derivatives $\dfrac{\partial}{\partial x_j}\mathcal{L}f(t,x)$, $j\in \{1,\ldots, K\}$, exist almost surely with respect to the Lebesgue measure on $\R$ and are bounded by $c$ (cf. e.g. \cite[Theorem 5.2.6]{Bogachev}). Since the integrand in \eqref{eq_Skorohod_integral_opt_approx} is a smooth functional of a multivariate Gaussian distribution with a continuous density, we conclude from \eqref{eq_thm_Ito_formula_Skorohod4} the uniform bound
\begin{align}
\displaybreak[0]
&\max_{j\in \{1,\ldots, K\}}\sum_{M \geq 0} \ex\left[\left(\sum_{|\bar{l}| = M} a'_{\bar{l}}(s) l_j W_{t_j}^{\diamond l_j-1} \diamond \ldots \diamond W_{t_K}^{\diamond l_K}\right)^2\right] \nonumber\\
&= \max_{j\in \{1,\ldots, K\}} \sum_{M \geq 0} \ex\left[\left(\dfrac{\partial}{\partial x_j} \sum_{|\bar{l}| = M} a'_{\bar{l}}(s) W_{t_1}^{\diamond l_1} \diamond \ldots \diamond W_{t_K}^{\diamond l_K}\right)^2\right]\nonumber\\
\displaybreak[0]
&=\max_{j\in \{1,\ldots, K\}}\ex\left[\left(\dfrac{\partial}{\partial x_j}\mathcal{L}f(s,t_2,\ldots, t_K,W_{t_1}, \ldots, W_{t_K})\right)^2\right] \leq c^2.\label{eq_thm_optimal_approx_simple_Skorohod0c_2}
\end{align}

\textit{Step 2}: 

Due to the Wiener chaos decomposition of the integrand in \eqref{eq_thm_Ito_formula_Skorohod1}, the problem is reduced to the orthogonal Wiener chaoses
\begin{equation}\label{eq_wiener_chaos_opt_approx_Skorohod_int}
u^M_{s} := \sum_{|\bar{l}| = M} a_{\bar{l}}(s) W_s^{\diamond l_1} \diamond \ldots \diamond W_{t_K}^{\diamond l_K}
\end{equation}
for fixed $M \in \N$. Thanks to the integration by parts formula \eqref{eq_thm_Ito_formula_Skorohod3}, Proposition \ref{prop_Wick_conditional_expectation}, Proposition \ref{prop_S_transform_properties} (iv) and the expansion
\begin{equation}\label{eq_Wick_power_difference_expansion_Brownian_bridge}
W_t^{\diamond k} - (W_{t}^{\lin})^{\diamond k} = B^n_t \diamond \sum\limits_{j=1}^{k} W_t^{\diamond k-j} \diamond (W_{t}^{\lin})^{\diamond j-1}, 
\end{equation}
we have
\begin{align}
\displaybreak[0]
&\int_{0}^{1}u^M_{s} dW_s - \ex[\int_{0}^{1}u^M_{s} dW_s|\P_n]\nonumber\\
&=\sum_{|\bar{l}| = M} \frac{1}{l_1+1}\left(\int_{0}^{1}a'_{\bar{l}}(s) W_s^{\diamond l_1+1}ds - \ex[\int_{0}^{1}a'_{\bar{l}}(s) W_s^{\diamond l_1+1}ds|\P_n]\right) \diamond \ldots \diamond W_{t_K}^{\diamond l_K}\nonumber\\
\displaybreak[0]
&=\int_{0}^{1} \sum_{|\bar{l}| = M} \frac{1}{l_1+1} a'_{\bar{l}}(s)\left(W_s^{\diamond l_1+1} - (W_{s}^{\lin})^{\diamond l_1+1}\right)ds \diamond  \ldots \diamond W_{t_K}^{\diamond l_K}\nonumber\\
\displaybreak[0]
&=\int_{0}^{1} B^n_s \diamond \sum_{|\bar{l}| = M} a'_{\bar{l}}(s) \frac{1}{l_1+1} \sum\limits_{l_0=1}^{l_1+1} W_s^{\diamond l_1+1-l_0} \diamond (W_{s}^{\lin})^{\diamond l_0-1} \diamond  \ldots \diamond W_{t_K}^{\diamond l_K} ds\nonumber\\
&=\sum_{i=1}^{n}\int_{(i-1)/n}^{i/n} B^n_s \diamond \sum_{|\bar{l}| = M} a'_{\bar{l}}(s) \frac{1}{l_1+1} \sum\limits_{l_0=1}^{l_1+1} W_s^{\diamond l_1+1-l_0} \diamond (W_{s}^{\lin})^{\diamond l_0-1} \diamond  \ldots \diamond W_{t_K}^{\diamond l_K} ds.\label{eq_thm_optimal_approx_simple_Skorohod1}
\end{align}

Now we show that instead of \eqref{eq_thm_optimal_approx_simple_Skorohod1} it suffices to consider the simpler terms
\begin{align}
\sum_{i=1}^{n}\int_{(i-1)/n}^{i/n} B^n_s \diamond \sum_{|\bar{l}| = M} a'_{\bar{l}}(s) W_s^{\diamond l_1} \diamond \ldots \diamond W_{t_K}^{\diamond l_K} ds,\label{eq_thm_optimal_approx_simple_Skorohod1_simple_error} 
\end{align}
which will be closer to $\mathcal{L} u^M_{s}$ in the asserted constant. The difference of \eqref{eq_thm_optimal_approx_simple_Skorohod1_simple_error} and \eqref{eq_thm_optimal_approx_simple_Skorohod1} is based on the following random variables for $l_1>0$, which are reformulated via \eqref{eq_Wick_power_difference_expansion_Brownian_bridge}:
\begin{align}
 &W_s^{\diamond l_1}  - \frac{1}{l_1+1} \sum\limits_{l_0=1}^{l_1+1} W_s^{\diamond l_1+1-l_0} \diamond (W_{s}^{\lin})^{\diamond l_0-1}\nonumber\\
 &= \frac{1}{l_1+1} \sum\limits_{l_0=1}^{l_1+1} W_s^{\diamond l_1+1-l_0} \diamond \left( W_s^{\diamond l_0-1}- (W_{s}^{\lin})^{\diamond l_0-1}\right) = B^n_s \diamond \frac{1}{l_1+1} \sum_{l_0=1}^{l_1+1} \sum_{m=1}^{l_0-1} W_s^{\diamond l_1-m} \diamond (W_s^{\lin})^{\diamond m-1}\nonumber\\
 &=B^n_s \diamond \sum_{m=1}^{l_1} \left(\frac{l_1+1-m}{l_1+1}\right) W_s^{\diamond l_1-m} \diamond (W_s^{\lin})^{\diamond m-1}.\label{eq_thm_optimal_approx_simple_Skorohod1_error_rv}
\end{align}
All $L^2$-norms are finally based on the computations in Remark \ref{remark_norm_computations}. Making use of \eqref{eq_Wick_product_of_Gaussian_inner_product} and \eqref{eq_covariances1}-\eqref{eq_covariances4}, for all covariances involved and necessarily $|\bar{l}| = |\bar{l}'|$, we observe
\begin{align}
\displaybreak[0]
&\ex[(B^n_s \diamond W_{t_1}^{\diamond l_1} \diamond  \ldots \diamond W_{t_K}^{\diamond l_K}) (B^n_{s'} \diamond W_{t_1'}^{\diamond l_1'} \diamond  \ldots \diamond W_{t_K'}^{\diamond l_K'})]\nonumber\\ 
&= \ex[B^n_s B^n_{s'}] \ex\left[\left(W_{t_1}^{\diamond l_1}\diamond  \ldots \diamond W_{t_K}^{\diamond l_K}\right) \left(W_{t_1'}^{\diamond l_1'}\diamond  \ldots \diamond W_{t_K'}^{\diamond l_K'}\right)\right]\nonumber\\
&\qquad + \sum\limits_{\substack{\lfloor ns'\rfloor = \lfloor n t_j\rfloor\\
\lfloor ns\rfloor = \lfloor n t_{j'}'\rfloor}}\ex[B^n_{s'} W_{t_j}] \ex[B^n_{s} W_{t_{j'}'}] l_j l_{j'}'\ex\left[\left(W_{t_j}^{\diamond l_j-1} \diamond \ldots \diamond W_{t_K}^{\diamond l_K}\right) \left(W_{t_{j'}'}^{\diamond l_{j'}'-1} \diamond \ldots \diamond W_{t_K'}^{\diamond l_K'}\right)\right].
\label{eq_thm_optimal_approx_simple_Skorohod1_1}
\end{align}
Similarly, due to \eqref{eq_covariances1}-\eqref{eq_covariances4}, extracting only one covariance, for \eqref{eq_thm_optimal_approx_simple_Skorohod1_error_rv} we have the simple upper bounds
\begin{align}
&\ex\left[\left(B^n_s \diamond \sum_{m=1}^{l_1} \left(\frac{l_1+1-m}{l_1+1}\right) W_s^{\diamond l_1-m} \diamond (W_s^{\lin})^{\diamond m-1}\right) \left(B^n_{s'} \diamond \sum_{m=1}^{l_1} \cdots \right)\right]\nonumber\\ 
&= \ex[B^n_s B^n_{s'}] \ex\left[\left(\sum_{m=1}^{l_1} \cdots \right) \left(\sum_{m=1}^{l_1} \cdots\right)\right]\nonumber\\
&\qquad + \ex[B^n_{s'} W_{s}] \ex\left[\left(\sum_{m=1}^{l_1} \frac{(l_1+1-m)(l_1-m)}{(l_1+1)} B^n_s \diamond W_s^{\diamond l_1-m-1} \diamond (W_s^{\lin})^{\diamond m-1} \right) \left(\sum_{m=1}^{l_1} \cdots\right)\right]\nonumber\\
&\leq \ex[B^n_s B^n_{s'}] l_1^2 (l_1-1)! (s \wedge s')^{l_1-1} + \ex[B^n_{s'} W_{s}] (l_1-1) \, l_1 (l_1-1)! (s \wedge s')^{l_1-1}\nonumber\\
&\leq 2\ex[B^n_s B^n_{s'}] l_1 l_1! (s \wedge s')^{l_1-1}.\label{eq_thm_optimal_approx_simple_Skorohod1_error_norm}
\end{align}
Thus, in covariances of Wick products of the type \eqref{eq_thm_optimal_approx_simple_Skorohod1_1} including the random variables \eqref{eq_thm_optimal_approx_simple_Skorohod1_error_rv}, \eqref{eq_thm_optimal_approx_simple_Skorohod1_error_norm} shows that the term \eqref{eq_thm_optimal_approx_simple_Skorohod1_error_rv} behaves in upper bounds as a random variable of the type
\begin{align}\label{eq_thm_optimal_approx_simple_Skorohod1_error_norm_simp}
B^n_s \diamond l_1 W_{s}^{\diamond l_1-1}.
\end{align}
Hence, due to \eqref{eq_covariances1}-\eqref{eq_covariances4}, \eqref{eq_Wick_product_of_Gaussian_inner_product}, \eqref{eq_thm_optimal_approx_simple_Skorohod1_error_rv}-\eqref{eq_thm_optimal_approx_simple_Skorohod1_error_norm_simp}, the upper bound in \eqref{eq_thm_optimal_approx_simple_Skorohod0c_2} and the Cauchy-Schwarz inequality, for the $L^2$-norm of the difference of \eqref{eq_thm_optimal_approx_simple_Skorohod1_simple_error} and \eqref{eq_thm_optimal_approx_simple_Skorohod1} we conclude
\begin{align*}
&\ex\left[\left(\sum_{i=1}^{n}\int_{(i-1)/n}^{i/n} B^n_s \diamond \sum_{|\bar{l}| >0} a'_{\bar{l}}(s) \left(B^n_s \diamond \sum_{m=1}^{l_1} \frac{l_1+1-m}{l_1+1} W_s^{\diamond l_1-m} \diamond (W_s^{\lin})^{\diamond m-1}\right) \diamond \ldots \diamond W_{t_K}^{\diamond l_K} ds\right)^2\right]\nonumber\\ 
&\leq \sum_{M\geq 0}\ex\left[\left(\sum_{i=1}^{n}\int_{(i-1)/n}^{i/n} (B^n_s)^{\diamond 2} \diamond \sum_{|\bar{l}| = M} a'_{\bar{l}}(s) l_1 W_{s}^{\diamond l_1-1}\diamond  \ldots \diamond W_{t_K}^{\diamond l_K} ds\right)\right]\nonumber\\ 
&=\sum_{i,i'=1}^{n}\int_{(i-1)/n}^{i/n}\int_{(i'-1)/n}^{i'/n} \ex\left[(B^n_s)^{\diamond 2} (B^n_{s'})^{\diamond 2}\right] \ex\left[\left(\dfrac{\partial}{\partial x_1}\mathcal{L}f(s,t_2,\ldots, t_K,W_{t_1}, \ldots, W_{t_K})\right)^2\right]^{1/2}\nonumber\\
&\qquad \cdot \ex\left[\left(\dfrac{\partial}{\partial x_1}\mathcal{L}f(s',t_2,\ldots, t_K,W_{t_1}, \ldots, W_{t_K})\right)^2\right]^{1/2} ds \, ds'\nonumber\\ 
&\leq c^2 \sum_{i=1}^{n}\int_{(i-1)/n}^{i/n}\int_{(i-1)/n}^{i/n} 2 (\ex[B^n_s B^n_{s'}])^2 ds \, ds'\ \leq \frac{c^2}{24} n^{-3}.
\end{align*}
Thus it suffices to consider the simplified mean squared error via \eqref{eq_thm_optimal_approx_simple_Skorohod1_error_rv} as
\begin{align}\label{eq_thm_optimal_approx_simple_Skorohod1_error_new}
f_n^2 = \sum_{M\geq 0}\ex\left[\left(\sum_{i=1}^{n}\int_{(i-1)/n}^{i/n} B^n_s \diamond \sum_{|\bar{l}| = M} a'_{\bar{l}}(s) W_s^{\diamond l_1} \diamond \ldots \diamond W_{t_K}^{\diamond l_K} ds\right)^2\right].
\end{align}

\textit{Step 3}:

Now we conclude with the computation of the limit in the assertion. 
Hence, for the random variables in \eqref{eq_thm_optimal_approx_simple_Skorohod1_error_new}, we obtain
\begin{align}
\displaybreak[0]
f_n^2
&= \sum_{M\geq 0}\sum_{i=1}^{n}\int_{(i-1)/n}^{i/n} \int_{(i-1)/n}^{i/n} \ex[B^n_s B^n_{s'}]\ex\left[\left(\sum_{|\bar{l}| = M} a'_{\bar{l}}(s) W_{s}^{\diamond l_1}\diamond  \ldots \diamond W_{t_K}^{\diamond l_K}\right)\right.\nonumber\\
&\qquad \qquad \cdot \left.\left(\sum_{|\bar{l}'| = M} a'_{\bar{l}}(s') W_{s'}^{\diamond l_1'}\diamond  \ldots \diamond W_{t_K}^{\diamond l_K'}\right)\right]ds \, ds'\nonumber\\
\displaybreak[0]
&\ + \sum_{M\geq 0}\sum_{i,i'=1}^{n}\int_{(i-1)/n}^{i/n} \int_{(i'-1)/n}^{i'/n} \sum\limits_{\substack{\lfloor ns'\rfloor = \lfloor n t_j\rfloor\\
\lfloor ns\rfloor = \lfloor n t_{j'}\rfloor}}\ex[B^n_{s'} W_{t_j}] \ex[B^n_{s} W_{t_{j'}}] \nonumber\\
&\qquad  \cdot \ex\left[\left(\sum_{|\bar{l}| = M} a'_{\bar{l}}(s) l_j W_{t_j}^{\diamond l_j-1} \diamond \ldots \diamond W_{t_K}^{\diamond l_K}\right)\left(\sum_{|\bar{l}'| = M} a'_{\bar{l}'}(s') l_{j'}'W_{t_{j'}}^{\diamond l_{j'}'-1} \diamond \ldots \diamond W_{t_K}^{\diamond l_K'}\right)\right] ds \, ds'\nonumber\\
\displaybreak[0]
&= \sum_{i=1}^{n}\int_{(i-1)/n}^{i/n} \int_{(i-1)/n}^{i/n} \ex[B^n_s B^n_{s'}]\ex\left[\left(\mathcal{L} f(s,\ldots, W_{s}, \ldots)\right)\left(\mathcal{L} f(s',\ldots, W_{s'}, \ldots)\right)\right]ds \, ds'\nonumber\\
&\ + \sum_{i,i'=1}^{n}\int_{(i-1)/n}^{i/n} \int_{(i'-1)/n}^{i'/n} \sum\limits_{\substack{\lfloor ns'\rfloor = \lfloor n t_j\rfloor\\
\lfloor ns\rfloor = \lfloor n t_{j'}\rfloor}}\ex[B^n_{s'} W_{t_j}] \ex[B^n_{s} W_{t_{j'}}] \nonumber\\
&\qquad  \cdot \ex\left[\left(\dfrac{\partial}{\partial x_{j}}\mathcal{L} f(s,\ldots, W_{s}, \ldots)\right)\left(\dfrac{\partial}{\partial x_{j'}}\mathcal{L} f(s,\ldots, W_{s}, \ldots)\right)\right] ds \, ds'\nonumber\\
\displaybreak[0]
&=: X^n_1 + X^n_2.\label{eq_thm_optimal_approx_simple_Skorohod2}
\end{align}

For the first sum, via \eqref{eq_covariances5}, \eqref{eq_thm_optimal_approx_simple_Skorohod0abc} and the triangle inequality, we obtain
\begin{align}
\displaybreak[0]
&\left|X^n_1 - \frac{1}{12\, n^2} \left(\frac{1}{n}\sum_{i=1}^{n}\ex\left[\left(\mathcal{L} f((i-1)/n, t_2,\ldots, t_K, W_{(i-1)/n}, \ldots, W_{t_K})\right)^2\right]\right)\right|\nonumber\\
&=\left|X^n_1 - \sum_{i=1}^{n}\int_{(i-1)/n}^{i/n} \int_{(i-1)/n}^{i/n} \ex[B^n_s B^n_{s'}]\ex\left[\left(\sum_{M \geq 0}\mathcal{L} u^M_{(i-1)/n}\right)^2\right]ds \, ds'\right|\nonumber\\
&\leq  \sum_{i=1}^{n}\int_{(i-1)/n}^{i/n} \int_{(i-1)/n}^{i/n} \ex[B^n_s B^n_{s'}] c' n^{-(\ve \wedge 1/2)} ds \, ds' \leq \frac{c'}{12} n^{-(2+\ve\wedge 1/2)}.
\label{eq_thm_optimal_approx_simple_Skorohod3}
\end{align}
For the second term in \eqref{eq_thm_optimal_approx_simple_Skorohod2}, by the triangle inequality, the Cauchy-Schwarz inequality,  \eqref{eq_thm_optimal_approx_simple_Skorohod0c_2}, \eqref{eq_covariances3} and \eqref{eq_covariances5}, we have
\begin{align}\label{eq_thm_optimal_approx_simple_Skorohod4}
\displaybreak[0] 
|X^n_2| \leq c^2 \sum_{i,i'=1}^{n}\int_{(i-1)/n}^{i/n} \int_{(i'-1)/n}^{i'/n} \sum\limits_{\substack{\lfloor ns'\rfloor = \lfloor n t_j\rfloor\\
\lfloor ns\rfloor = \lfloor n t_{j'}\rfloor}}\ex[B^n_{s'} W_{t_j}] \ex[B^n_{s} W_{t_{j'}}]  ds \, ds' \leq c^2 \frac{1}{12} n^{-3}.
\end{align}

Due to Step 1, \eqref{eq_thm_optimal_approx_simple_Skorohod1_error_new}, \eqref{eq_thm_optimal_approx_simple_Skorohod2}-\eqref{eq_thm_optimal_approx_simple_Skorohod4} and
\[
\lim_{n \rightarrow \infty}\frac{1}{n}\sum_{i=1}^{n}\ex\left[\left(\mathcal{L} f((i-1)/n, t_2,\ldots, t_K, W_{(i-1)/n}, \ldots, W_{t_K})\right)^2\right] =C,
\]
we conclude the assertion.
\end{proof}

\begin{remark}\label{remark2_main_thm}
$\left.\right.$
\begin{enumerate}
 \item We notice by \eqref{eq_thm_optimal_approx_simple_Skorohod000}, \eqref{eq_thm_optimal_approx_simple_Skorohod0c_2} in the proof that it suffices to assume the weaker conditions in Theorem \ref{thm_optimal_approx_simple_Skorohod} :
\begin{align*}
\displaybreak[0] 
&\ex[|\mathcal{L}f(s, t_2,\ldots, t_K,W_{u},\ldots, W_{t_K}) - \mathcal{L}f(t, t_2,\ldots, t_K,W_{v},\ldots, W_{t_K})|^2]\leq c(|s-t|+|u-v|)^{\ve},\\
&\max_{j\in \{1,\ldots, K\}}\ex\left[\left(\dfrac{\partial}{\partial x_j}\mathcal{L}f(s,t_2,\ldots, t_K,W_{t_1}, \ldots, W_{t_K})\right)^2\right] \leq c,
\end{align*}
for some constants $c,\ve>0$ instead of (L) and (HG).
\item For example, by \eqref{eq_Wick_sine} in Remark \ref{remark_sine}, Theorem \ref{thm_optimal_approx_simple_Skorohod} applies on the integrand
\[
\sin\left(W_s + \sum_{i=2}^{K} W_{\tau_i}\right)_{s \in [0,1]} 
\]
with
\begin{align*}
&f^{\diamond}(s,\tau_2,\ldots, \tau_K, W_s,W_{\tau_2},\ldots, W_{\tau_K})= \exp\left(-\frac{1}{2}\|W_s + \sum_{i=2}^{K} W_{\tau_i}\|_{L^2}^2\right) \sin^{\diamond}\left(W_s + \sum_{i=2}^{K} W_{\tau_i}\right)
\end{align*}
and
\begin{align*}
&\mathcal{L} f^{\diamond}(s,\tau_2,\ldots, \tau_K, W_s,W_{\tau_2},\ldots, W_{\tau_K}) \\
&= \left((\dfrac{\partial}{\partial s} + \sum_{i=2}^{K} \dfrac{\partial}{\partial \tau_j})\exp\left(-\frac{1}{2}\|W_s + \sum_{i=2}^{K} W_{\tau_i}\|^2\right)\right) \sin^{\diamond}\left(W_s + \sum_{i=2}^{K} W_{\tau_i}\right).
\end{align*}
One can check the assumptions in (i) above by simple computations and Proposition \ref{prop_Wickexp_properties} (iii).

\end{enumerate}
\end{remark}

\end{document}